\documentclass[11pt]{amsart}
\usepackage[colorlinks = true,
            linkcolor = red,
            urlcolor  = magenta,
            citecolor = black,
            anchorcolor = blue]{hyperref}
\usepackage{color}
\usepackage[alphabetic]{amsrefs}
\usepackage{tikz}
\usetikzlibrary{arrows}
\usepackage[all]{xy}
\usepackage{graphicx}
\usepackage{wrapfig}
\usepackage{geometry}                
\usepackage{graphicx}
\usepackage{amssymb}
\usepackage{epstopdf}
\usepackage{mathrsfs}
\usepackage[mathcal]{eucal}

\makeatletter \@addtoreset{equation}{section} \makeatother

\renewcommand{\eprint}[1]{\href{https://arxiv.org/abs/#1}{#1}}

\BibSpec{article}{%
    +{}  {\PrintAuthors}                {author}
    +{,} { \textit}                     {title}
    +{.} { }                            {part}
    +{:} { \textit}                     {subtitle}
    +{,} { \PrintContributions}         {contribution}
    +{.} { \PrintPartials}              {partial}
    +{,} { }                            {journal}
    +{}  { \textbf}                     {volume}
    +{}  { \PrintDatePV}                {date}
    +{,} { \issuetext}                  {number}
    +{,} { \eprintpages}                {pages}
    +{,} { }                            {status}
    +{,} { \DOI}                   {doi}
    +{,} { \eprint}        {eprint}
    +{}  { \parenthesize}               {language}
    +{}  { \PrintTranslation}           {translation}
    +{;} { \PrintReprint}               {reprint}
    +{.} { }                            {note}
    +{.} {}                             {transition}
    +{}  {\SentenceSpace \PrintReviews} {review}
}

\def\O {{\mathcal{O}}}

\newcommand{\ard}{{\mathbf{d}}}

\def\fp{ {\textbf{p}}  }

\makeatletter
\newcommand{\ellSN}{\mathop{\operator@font sn}\nolimits}
\newcommand{\ellCN}{\mathop{\operator@font cn}\nolimits}
\newcommand{\ellDN}{\mathop{\operator@font dn}\nolimits}
\newcommand{\ellAM}{\mathop{\operator@font am}\nolimits}
\newcommand{\ellK}{\mathop{\smash{\operator@font K}\vphantom{a}}\nolimits}
\newcommand{\ellE}{\mathop{\smash{\operator@font E}\vphantom{a}}\nolimits}
\makeatother

\ifx\genfrac\sdflkaj

\else

\fi

\newcommand{\beq}{\begin{equation}}
\newcommand{\eeq}{\end{equation}}

\makeatletter
\def\mr@ignsp#1 {\ifx\:#1\@empty\else #1\expandafter\mr@ignsp\fi}%
\newcommand{\multiref}[1]{\begingroup
\xdef\mr@no@sparg{\expandafter\mr@ignsp#1 \: }%
\def\mr@comma{}%
\@for\mr@refs:=\mr@no@sparg\do{\mr@comma\def\mr@comma{,}\ref{\mr@refs}}%
\endgroup}
\makeatother

\newcommand{\hypref}[2]{\ifx\href\asklfhas #2\else\href{#1}{#2}\fi}
\newcommand{\Secref}[1]{Section~\multiref{#1}}
\newcommand{\secref}[1]{Sec.~\multiref{#1}}

\renewcommand{\eqref}[1]{(\multiref{#1})}

\def\[{\begin{equation}}
\def\]{\end{equation}}
\def\<{\begin{eqnarray}}
\def\>{\end{eqnarray}}

\newtheorem{theorem}{Theorem}[section]
\newtheorem{lemma}[theorem]{Lemma}
\newtheorem{proposition}[theorem]{Proposition}
\newtheorem{corollary}[theorem]{Corollary}

\newtheorem{definition}[theorem]{Definition}

\ifx\href\asklfhas\newcommand{\href}[2]{#2}\fi

\title[qKZ/tRS duality via quantum K-theoretic counts]{qKZ/tRS duality via quantum K-theoretic counts}

\author{Peter Koroteev}
\author{Anton M. Zeitlin}

\address{\newline 
Peter Koroteev\newline
Department of Mathematics,\newline
University of California at Davis,\newline
Mathematical Sciences Building,\newline
One Shields Ave,\newline
University of California,\newline
Davis, CA 95616;\newline
Department of Mathematics,\newline
University of California Berkeley,\newline
970 Evans Hall,\newline
University of California,\newline
Berkeley, CA 94720-3840,\newline
\href{mailto:pkoroteev@math.berkeley.edu}{pkoroteev@math.berkeley.edu}\newline
\url{https://math.berkeley.edu/~pkoroteev}
}

\address{
\newline
Anton M. Zeitlin,\newline
Department of Mathematics,\newline
Louisiana State University,\newline
Baton Rouge LA 70803-4918;\newline
IPME RAS, V.O. Bolshoj pr., 61, \newline 
199178, St. Petersburg\newline
\href{mailto:zeitlin@lsu.edu}{zeitlin@lsu.edu},\newline
\url{http://math.lsu.edu/~zeitlin}
}

\begin{document}
\maketitle

\begin{abstract}
We show that normalized quantum K-theoretic vertex functions for cotangent bundles of partial flag varieties are the eigenfunctions of quantum trigonometric Ruijsenaars-Schneider (tRS) Hamiltonians. Using recently observed relations between quantum Knizhnik-Zamolodchikov (qKZ) equations and tRS integrable system we derive a nontrivial identity for vertex functions with relative insertions.
\end{abstract}

\tableofcontents

\section{Introduction}\label{Sec:Intro}
The fundamental work of Givental and Lee \cite{2001math8105G} on quantum K-theory opened a path to study this subject from the perspective of quantum integrable systems. The example they studied in detail is quantum K-theory of flag varieties, which turned out to be related to the many-body system known as relativistic Toda chain. They introduced certain generating K-theory-valued functions, containing  the corresponding enumerative data, known as Givental's $J$-functions, which were found to be eigenfunctions for the first nontrivial relativistic Toda Hamiltonian.

Recent progress in understanding the relationship between enumerative geometry and integrability, inspired by the physics papers of Nekrasov and Shatashvili \cite{Nekrasov:2009ui,Nekrasov:2009uh}, was achieved by Okounkov and his collaborators \cite{Braverman:2010ei,2012arXiv1211.1287M,Okounkov:2015aa}, in the study of enumerative geometry of symplectic resolutions (see e.g. \cite{Okounkov:2018,Okounkov:2017}).  

For a wide class of symplectic resolutions, known as Nakajima quiver varieties, an explicit relation between quantum K-theoretic counts, based on a theory of quasimaps developed in \cite{Ciocan-Fontanine:2011tg}, and q-difference equations, known to representation theory community as quantum Knizhnik-Zamolodchikov (qKZ) equations \cite{Frenkel:92qkz} and the related dynamical equations was discovered in \cite{Okounkov:2015aa,Okounkov:2016sya}. 

We would like to mention that the Okounkov's approach to quantum K-theory and enumerative computations is substantially different from the original Givental's approach and while in this paper we will establish  correspondences in certain explicit examples of K-theoretic counts, as we have done already on the level of quantum rings \cite{Koroteev:2017}, the precise relation between the two approaches still requires further understanding. 

The relationship between quantum K-theoretic counts and quantum difference equations lead to several new results, in particular to proper understanding of the ring structure of K-theory of Nakajima variety \cite{Pushkar:2016qvw}, \cite{Koroteev:2017} as Bethe algebra for various kinds of spin chain models, explicit combinatorial formulae for quantum multiplication and many other observations. A very important recent discovery of Aganagic and Okounkov \cite{Aganagic:2017be} is the explicit construction and geometric interpretation of solutions of qKZ equations. 

At the same time, close kinship between integrable models constructed using quantum groups and many-body systems of Toda, Ruijsenaars-Schneider and Calogero-Moser type was inspired by work of Cherednik \cite{Cherednik:1991mg} and Matsuo \cite{Matsuo1992} resulting in numerous follow-up contributions, see e.g. \cite{Mukhin:2009,Mukhin:,Zabrodin: 2014,Zabrodin:,Zabrodin:2017} which focused their study on integrable models as well as in physics literature \cite{Gaiotto:2013bwa,Bullimore:2015fr}.

Recently the quantum K-theory rings of cotangent bundles to flag varieties, as well as flag varieties themselves, were reinterpreted as algebras of functions on the Lagrangian subvarieties in the phase spaces of {\it classical} trigonometric Ruijsenaars-Schneider (tRS) integrable system and relativistic Toda lattice correspondingly \cite{Koroteev:2017}. 

Motivated by the renewed interest to difference equations in quantum K-theory (see e.g. \cite{Anderson:2017}), this paper is aiming to provide a proper analogue of Givental's $J$-functions for Nakajima quiver varieties (we shall focus on cotangent bundles to partial flag varieties of type A) and establish their relation to solutions of quantum difference Knizhnik-Zamolodchikov equations. We would like to reiterate that K-theory vertex functions which will be computed in this paper are defined differently than Givental's $J$-functions (quasimaps vs. stable maps).

The structure of this paper is as follows. In \Secref{Sec:QKZNak} we discuss necessary background in quantum K-theory, in particular in quantum K-theory of Nakajima varieties, bare and capped vertex functions, fusion matrix, and the relation to qKZ equations along the lines of \cite{Okounkov:2015aa,Okounkov:2016sya,Aganagic:2017be}, providing detailed references to the corresponding sections and statements in \textit{loc. cit}. 

In \Secref{sec:VertFunc} we discuss a particular integral representation for bare vertex functions of cotangent bundles for partial flag varieties, which were studied in  \cite{Koroteev:2017} (see \cite{Pushkar:2016qvw} for the detailed study of cotangent bundles to Grassmannians).

In \Secref{Sec:tRS} we prove that the normalized vertex function without insertions is the eigenfunction of trigonometric Ruijsenaars-Schneider difference operators (also known as Macdonald operators). To the best of our knowledge, \Secref{Sec:tRS} of this manuscript contains the first mathematical proof of this fact. The paper by Givental and Lee on quantum K-theory of complete flag varieties states that their $J$-functions satisfy quantum difference Toda (qToda) relations. We shall discuss the connection between the tRS and qToda systems.  

Finally, in \Secref{Sec:qKZtRS} we return to qKZ equations. First we shall discuss the trigonometric version of 
the duality between solutions of qKZ equations and the tRS eigenfunctions, thereby generalizing \cite{Zabrodin:2017}. Then, by combining that with the results of \Secref{Sec:tRS}, we obtain a nontrivial identity between normalized vertex function and vertex functions with relative insertions.
This provides a geometric meaning for the qKZ/tRS correspondence.
\\

\noindent \textbf{Acknowledgements.}
We would like to thank P. Pushkar and A. Okounkov for helpful discussions.
The work of A.M.Zeitlin is partially supported by Simons Collaboration Grant, Award ID: 578501. 
P.K. acknowledges support of IH\'ES and funding from the European
Research Council (ERC) under the European Union's Horizon 2020
research and innovation program (QUASIFT grant agreement 677368).

\section{Quantum Knizhnik-Zamolodchikov equations and Nakajima varieties}\label{Sec:QKZNak}
In this section we provide a brief review the Okounkov's quasimap approach\footnote{We note that quasimap approach to quantum K-theory introduced by Okounkov is different from the original approach of Givental. While certain K-theoretic counts do coincide (see e.g. \cite{Koroteev:2017}), the exact relationship between such approaches is still unknown.} to the quantum equivariant K-theory of Nakajima quiver varieties. One of the important advantages of this approach is the emergence of quantum Knizhnik-Zamolodchikov equation and its dynamical counterpart. In this section we present a short review of quantum K-theory, where we will introduce all necessary objects to guide the reader towards qKZ equations, which will be our main focus. 
This is unlike our previous work \cite{Pushkar:2016qvw}, \cite{Koroteev:2017}, where the emphasis was made on the dual dynamical equation. For the comfort of the reader, we provide detailed references to the appropriate sections of Okounkov's lectures \cite{Okounkov:2015aa} as well as to recent paper by Aganagic and Okounkov \cite{Aganagic:2017be}. We also encourage the reader to consult introductory part of \cite{Pushkar:2016qvw}. 

\subsection{Nakajima quiver varieties and their equivariant K-theory}\label{eq:AtypeNakajima}
A quiver is a collection of vertices and oriented edges connecting them. A framed quiver is a quiver, where the set of vertices is doubled, and each of the vertices in the added set has an edge going from it to the vertex, whose copy it is. 

A representation of a framed quiver is a set of vector spaces $V_i,W_i$, where $V_i$ correspond to original vertices, and $W_i$ correspond to their copies, together with a set of morphisms between these vertices, corresponding to edges of the quiver.

For a given framed quiver, let $M:=\text{Rep}(\mathbf{v},\mathbf{w})=\sum_{i\in I}\text{Hom}({W}_i,{V}_i)\oplus\sum_{i,j\in I}{ Q}_{ij}\otimes \text{Hom}({V}_i,{V}_j)$ denote the linear space of quiver representation with dimension vectors $\mathbf{v}$ and $\mathbf{w}$, where $\mathbf{v}_i=\text{dim}\ V_i$, $\mathbf{w}_i=\text{dim}\ W_i$ and $Q_{ij}$ stands for the incidence matrix of the quiver, i.e. the number of edges between vertices $i$ and $j$. Then the group $G=\prod_i GL(V_i)$ acts on this space in an obvious way. Thus we have $\mu :T^*M\to \text{Lie}{(G)}^{*}$, to be the moment map and let $L(\mathbf{v},\mathbf{w})=\mu^{-1}(0)$ be the zero locus of the moment map.

The Nakajima variety $X$ corresponding to the quiver is an algebraic symplectic reduction
$$
X=N(\mathbf{v},\mathbf{w})=L(\mathbf{v},\mathbf{w})/\!\! /_{\theta}G=L(\mathbf{v},\mathbf{w})_{ss}/G,
$$
depending on a choice of stablity parameter $\theta\in {\mathbb{Z}}^I$ (see \cite{Ginzburg:2009} for a detailed definition). The group
$$\prod_{i.j} GL(Q_{ij})\times
\prod_i GL({W}_i)\times \mathbb{C}^{\times}_\hbar$$
acts as automorphisms of $X$, coming form its action on $\text{Rep}(\mathbf{v},\mathbf{w})$. Here  $\mathbb{C}^{\times}_{\hbar}$ scales cotangent directions with weight $\hbar$ and therefore symplectic form with weight $\hbar^{-1}$. Let us denote by $\mathsf{T}$ the maximal torus of this group.

Our main example will be  the space of partial flags, i.e. we are considering the following quiver of type $A_n$ with single framing vertex:

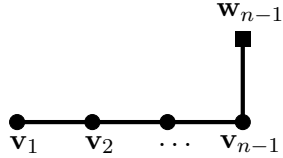
\vspace{0.1in}
\begin{center}
\begin{tikzpicture}
\draw [ultra thick] (0,0) -- (3,0);
\draw [ultra thick] (3,1) -- (3,0);
\draw [fill] (0,0) circle [radius=0.1];
\draw [fill] (1,0) circle [radius=0.1];
\draw [fill] (2,0) circle [radius=0.1];
\draw [fill] (3,0) circle [radius=0.1];
\node (1) at (0.1,-0.3) {$\mathbf{v}_{1}$};
\node (2) at (1.1,-0.3) {$\mathbf{v}_2$};
\node (3) at (2.1,-0.3) {$\ldots$};
\node (4) at (3.1,-0.3) {$\mathbf{v}_{n-1}$};
\fill [ultra thick] (3-0.1,1) rectangle (3.1,1.2);
\node (5) at (3.1,1.45) {$\mathbf{w}_{n-1}$};
\end{tikzpicture}
\end{center}
\vspace{0.1in}

The stability condition is chosen so that maps $W_{n-1}\to V_{n-1}$ and $V_i\to V_{i-1}$ are surjective. For the variety to be non-empty the sequence $\mathbf{v}_{1},\ldots ,\mathbf{v}_{n-1}, \mathbf{w}_{n-1}$ must be non-decreasing. The fixed points of this Nakajima quiver variety and the stability conditions are classified by chains of subsets $V_{1}\subset \ldots \subset V_{n-1}\subset W_{n-1}$.

For a Nakajima quiver variety $X$ one can define a set of tautological bundles 
$$
\mathcal{V}_i=L({\bf v}, {\bf w})_{ss}\times V_i/G,\quad  \mathcal{W}_i=L({\bf v}, {\bf w})_{ss}\times W_i/G.
$$
From this construction it follows that all bundles $\mathcal{W}_i$ are topologically trivial.  It is known that tensorial polynomials of these bundles and their duals generate $K_{\mathsf{T}}(X)$.  For more details regarding quiver varieties, one can consult with e.g. \cite{Ginzburg:2009},  introduction to \cite{2012arXiv1211.1287M} or Section 4 of \cite{Okounkov:2015aa}.

\subsection{Quasimaps, nonsingular and relative conditions}
The proper reference for this subsection would be Sections 4.3. and 6 of \cite{Okounkov:2015aa}. 

\begin{definition}
A quasimap $f$ from $\mathcal{C}\cong \mathbb{P}^1$ to $X$
\beq
f:{\mathcal{C}}\dashrightarrow X\nonumber
\eeq
Is a collection of vector bundles $\mathscr{V}_i$ on $\mathcal{C}$ of ranks $\mathbf{v}_i$ together with a section of the bundle
\beq
f\in H^0(\mathcal{C},\mathscr{M}\oplus \mathscr{M}^{\ast}\otimes \hbar),\nonumber
\eeq

satisfying $\mu =0$, where
$$
{\mathscr{M}}=\sum_{i\in I}Hom(\mathscr{W}_i,\mathscr{V}_i)\oplus\sum_{i,j\in I}
{\mathscr Q}_{ij}\otimes Hom(\mathscr{V}_i,\mathscr{V}_j),
$$
so that $\mathscr{W}_i$ and ${\mathscr Q}_{ij}$ are trivial bundles of rank $\mathbf{w}_i$ and $\mu$ is the moment map. Here $\hbar$ is a trivial line bundle with weight $\hbar$ introduced to have the action of $\mathsf{T}$ on the space of quasimaps.
The degree $\bf{d}$ of a quasimap is a the vector of degrees of bundles $\mathscr{V}_i$.
\end{definition}
We say that a quasimap is stable if 
$f(p)$ is also a stable (belongs to the orbits giving rise to Nakajima variety) for all but finitely many points of $\mathcal{C}$. If $f(p)$ is
not stable, we will say that the quasimap $f$ is {\it singular} at $p$.

For a point on the curve $p\in \mathcal{C}$ we have an evaluation map from the moduli space of stable quasimaps to the quotient stack $\text{ev}_p : \textsf{QM}^{\ard} \to L(\mathbf{v},\mathbf{w})/G$ defined by
$\text{ev}_p(f)=f(p)$. Note that the quotient stack contains $X$ as an open subset corresponding to locus of semistable points: $$X={\mu}_{ss}^{-1}(0)/G\subset L(\mathbf{v},\mathbf{w})/G.$$
A quasimap $f$ is called nonsingular at $p$ if $f(p)\subset X$. 

In short, we conclude that the open subset ${{\textsf{QM}}^\ard}_{\text{nonsing  p}}\subset {{\textsf{QM}}^\ard}$ of stable quasimaps of degree $\bf{d}$\footnote{Degree {\bf d} of a quasimap is understood as the vector of degrees of bundles $\mathscr{V}_i$, and thus can be considered as an element of $H_2(X,\mathbb{Z}).$}, nonsingular at the given point $p$,
is endowed with a natural evaluation map:
\beq
{{\textsf{QM}}^\ard}_{\text{nonsing }\, p} \stackrel{{\text{ev}}_p}{\longrightarrow} X
\eeq
which sends a quasimap to its value at $p$. The moduli space of relative quasimaps ${{\textsf{QM}}^\ard}_{\text{relative} \, p}$ is a resolution of ${\text{ev}}_p$ (or compactification), meaning we have a commutative diagram:
\begin{center}
\begin{tikzpicture}[node distance =5.1em]

  \node (rel) at (2.5,1.5) {${{\textsf{QM}}^\ard}_{\text{relative}\, p}$};
  \node (nonsing) at (0,0) {${{\textsf{QM}}^\ard}_{\text{nonsing}\,  p}$};
  \node (X) at (5,0) {$X$};
  \draw [->] (nonsing) edge node[above]{$\text{ev}_p$} (X);
  \draw [->] (rel) edge node[above]{$\widetilde{\text{ev}}_p$} (X);
  \draw [right hook->] (nonsing) edge  (rel);
    \end{tikzpicture}
\end{center}
with a proper evaluation map $\widetilde{\text{ev}}_p$ from ${{\textsf{QM}}^\ard}_{\text{relative}\, p}$ to $X$. The construction of this resolution and the moduli space of relative quasimaps is explained in \cite{Okounkov:2015aa}. The main idea of this construction is to allow the base curve to change in cases, when the relative point becomes singular. When this happens we replace the relative point by a chain of non-rigid projective lines, such that the endpoint and all the nodes are not singular. Similarly, for nodal curves, we do not allow the singularity to be at the node, and if that happens we instead paste in a chain of non-rigid projective lines.

These moduli spaces have a natural action of maximal torus $\mathsf{T}$, lifting its action from $X$. When there are at most two special (relative or marked) points and the original curve is ${\mathbb{P}}^1$ we extend $\mathsf{T}$ by additional torus $\mathbb{C}^{\times}_q$, which scales ${\mathbb{P}}^1$ such that the tangent space $T_{0} {\mathbb{P}}^1$ has character denoted by $q$. We call the full torus by $\mathsf{T}_q=\mathsf{T}\times \mathbb{C}^{\times}_q$.

The moduli spaces of quasimaps constructed above have a perfect deformation-obstruction theory \cite{Ciocan-Fontanine:2011tg}. This allows one to construct a tangent virtual bundle $T^{\textrm{vir}}$, a virtual structure sheaf $\widehat{\mathcal{O}}_{\rm{vir}}$ and a virtual canonical bundle \cite{Okounkov:2015aa}.

The spaces $\textsf{QM}^{\ard}_{\rm{nonsing}\, p_2}$ and $\textsf{QM}^{\ard}_{\rm{relative}\, p_2}$ admit the action of extra torus $\mathbb{C}^\times_q$ which scales curve $\mathcal{C}=\mathbb{P}^1$ keeping two points $p_1$ and $p_2$ fixed. 

One can also define the {\it twisted} version of quasimaps. 
Let $\sigma: \mathbb{C}^{\times}\to A$ be a cocharacter of subtorus $A=\text{Ker}\, \hbar\subset T$, preserving symplectic form, which determines the twist of a quasimap to $X$. As $T$ is acting on $W_i$,  $\sigma$ determines $\mathscr{W}_i$ over the curve $\mathcal{C}$ as bundles associated to $\mathscr{O}(1)$. 

For every fixed point $x\in X^{\sigma}$ there is a constant twisted quasimap $f(c)=x$ for any $c\in \mathcal{C}$. There is an important formula defining  the nontrivial degree of such quasimap (see Section 4.3.12 of \cite{Okounkov:2015aa}). Namely, given a fixed point of subtorus of $a\in A$, means that there is an action on line bundles $\mathcal{L}_i=\text{det}(\mathcal{V}_i)$ (where the tautological bundles $\mathcal{V}_i$ on $X$ correspond to the spaces $V_i$),  which produces a locally constant map
$$
\mu: X^{A}\to{Pic(X)}^{\vee}\otimes A^{\vee},
$$
which is compatible with the restrictions to the subgroups of $A$. Then 
$$
{\rm deg}(f\equiv x):=\langle \mu(x), -\otimes \sigma\rangle,
$$
where the pairing is between characters and cocharacters.

\subsection{Vertices with Descendants, Fusion Operator and qKZ Equation}\label{Sec:VertexFunctions}
The topics in this subsection are spread over Sections 7-10 of \cite{Okounkov:2015aa}, so we will provide more detailed references throughout.

The primary interest of this work is to study certain virtual quasimap counts, which are pushforwards of the twisted virtual structure sheaf $\widehat{{\O}}_{{\rm{vir}}}$ on the moduli spaces of stable quasimaps accompanied by extra insertions. Given a point $p$ on the curve, one can pick a fiber $\mathscr{V}_i|_{p}$. The collection of all such fibers forms a sheaf over the corresponding $\textsf{QM}$ space, as quasimap varies over this moduli space. Let $\tau$ be the tensor polynomial of $V_{\{i\}}, V^*_{\{i\}}$, representing the  K-theoretic class $\tau( \mathcal{V}_{\{i\}})$ over $X$. We denote by $\tau(\left.\mathscr{V}_{\{i\}}\right.)$ the corresponding class over the quasimap moduli space, obtained via the application of $\tau$ to the collection $\mathscr{V}_{\{i\}}|_{p}$. 

With that in mind, let us make the following definition (see Sections 7.2 and 7.4 of \cite{Okounkov:2015aa} as well as \cite{Pushkar:2016qvw}, \cite{Koroteev:2017} for details).
\begin{definition}
The elements
$$
V^{(\tau)}(z)=\sum\limits_{\ard=\vec{0}}^{\infty} z^{\ard} {\rm{ev}}_{p_2, *}\Big(\textsf{QM}^{\ard}_{{\rm{nonsing}} \, p_2},\widehat{{\O}}_{{\rm{vir}}} \otimes\tau (\left.\mathscr{V}_{\{i\}}\right|_{p_1}) \Big) \in  K_{\mathsf{T}_q}(X)_{loc}[[z]]
$$
$$
\widehat{V}^{(\tau)}(z)=\sum\limits_{\ard=\vec{0}}^{\infty} z^{\ard} {\rm{ev}}_{p_2, *}\Big(\textsf{QM}^{\ard}_{{\rm{relative}} \, p_2},\widehat{{\O}}_{{\rm{vir}}} \otimes\tau (\left.\mathscr{V}_{\{i\}}\right|_{p_1}) \Big) \in  K_{\mathsf{T}_q}(X)[[z]]
$$
are called bare and capped vertex with descendants $\tau$ correspondingly. 
\end{definition}

The space $\textsf{QM}^{\ard}_{\rm{nonsing}\, p_2}$ is not proper (the condition of non-singularity at a point is an open condition), but the set of $\mathsf{T}_q$-fixed points is, hence the bare vertex is singular at $q=1$. 

The following statement is known for capped vertices \cite{Smirnov:2016}.

\begin{theorem}
The power series $\widehat{V}^{(\tau)}(z)$ is a Taylor expansion of a rational function in quantum parameters $z$.
\end{theorem}

The operator which relates capped and bare vertices, is known as capping operator (see Section 7.4.4. of \cite{Okounkov:2015aa}) and is defined as the following class in the localized K-theory:
\begin{equation}
\Psi(z)=\sum\limits_{\ard=0}^{\infty}\, z^{\ard} {\rm ev}_{p_1,*}\otimes{\rm ev}_{p_2,*}
\Big( {\it \textsf{QM}}^{\ard}_{{{\rm{relative}\, p_1}} \atop {{\rm{nonsing}\, p_2}}},
\widehat{{\O}}_{{\rm{vir}}} \Big)
\in K^{\otimes 2}_{\mathsf{T}_q}(X)_{loc}[[z]]\,.
\end{equation}
Note that one can define the following bilinear form $(\cdot, \cdot)$ on $K_{\mathsf{T}}(X)$ via the equivariant Euler characteristic 
\begin{equation}
(\mathcal{F},\mathcal{G})= \chi\left(\mathcal{F}\otimes\mathcal{G}\otimes K^{-\frac{1}{2}}\right)\,,
\end{equation}
twisted by a root of the canonical class (the latter exists for Nakajima quiver varieties). This twisting is introduced due to the similar twisting of the structure sheaf $\widehat{{\O}}_{{\rm{vir}}}$ (see \cite{Okounkov:2015aa}). 
Because of this pairing, $\Psi(z)$ can be considered as an operator acting from the second to the first copy of $K_{\mathsf{T}_q}(X)_{loc}[[z]]$. 
In other words, the capped vertex with descendent $\tau$ is a result of applying of the capping operator to the bare vertex
\begin{equation}
\label{verrel}
\widehat{V}^{(\tau)}(z) = \Psi(z) {V}^{(\tau)}(z).
\end{equation}

The above operator satisfies quantum difference equations. For convenience one has to normalize the fusion operator as follows (see section 8.2 of \cite{Okounkov:2015aa}):
 \begin{equation}
\widehat{\Phi}(a,z)=\Psi(a,z)\Theta(a), \quad \Theta(a)=\varphi((1-q\hbar^{-1})T^{1/2}X)\Xi\,.
\label{hatphi}
\end{equation}
Here $T^{1/2}X=M-\sum_{i \in I}\text{End}(V_i,V_i)$ is the polarization bundle for the Nakajima variety and $\varphi: K(X)\to K(X)\otimes \mathbb{Z}[[q]]$ defined as $\varphi(\mathcal{F})=\prod (f_i)_{\infty}=
\Lambda^{\bullet}\Big(\mathcal{F}\otimes \sum_{i\ge 0} q^i\Big)$, where  
$f_i$ are Chern roots of $\mathcal{F}$. 
The operator $\Xi$ is defined as follows (see subsection 8.2.3 of \cite{Okounkov:2015aa}):
$$
\Xi=\exp\left(\frac{\xi(\log z_{\rm shifted}, \log {\rm t})}{\log q}\right),
$$
where map 
$$\xi: H^2(X,\mathbb{C})\otimes  \text{Lie}(T)\to \text{End}\,K(X^T)\otimes \mathbb{C},$$ 
extends by linearity the map of the logarithm of the tensor multiplication by a line bundle: $\mathcal{L}\to \log(\mathcal{L}\otimes -)$ to $H^2(X,\mathbb{C})$, while the shifted K\"ahler parameters are defined as follows:
$$
z^{\ard}_{\rm shifted}=z^{\ard} (-\hbar^{-1/2}q)^{-\langle \text{det}\, T^{1/2}, \ard\rangle}, ~{\rm where} ~{\bf d}\in H_2(X,\mathbb{Z}).
$$
The operator $\Xi$ satisfies the following difference equation:
$$
\Xi (q^{\sigma} a)=(z_{\rm shifted})^{\langle \mu(\cdot), -\otimes \sigma\rangle}\Xi ,
$$
where we recognize the term with a degree of twisted constant quasimap from the end of Section 2.2.
We can now write down the following important Theorem (see Proposition 6.5.30, Section 8.2.17,  Theorem 8.2.20 and Corollary 8.2.24 of \cite{Okounkov:2015aa}):
\begin{theorem}\label{Th:qkzgeom}
The normalized capping operator $\widehat{\Phi}(z)$ is the solution of the quantum difference equations:
$$
\widehat{\Phi}(q^{\sigma}a)={\bf S}_{\sigma}\widehat{\Phi}(a)
$$
where
$$
{\mathbf S}_{\sigma}=
\sum\limits_{\ard=0}^{\infty}\, z^{\ard}\,
{\rm ev}_{{p_1},*}\otimes {\rm ev}_{{p_2},*}\left(\textsf{QM}^{~\sigma,\ard}_{{{\rm{relative}\, p_1}} \atop {{\rm{relative}\, p_2}}}, \widehat{{\O}}_{{\rm{vir}}} \right) {\mathbf G}^{-1}\in K_{\mathsf{T}_q}^{\otimes 2}(X)_{loc}[[z]], 
$$
where $\textsf{QM}^{~\sigma,\ard}$ stands for the moduli space of $\sigma$-twisted quasimaps of degree $\ard$ and 
$\mathbf G$ is the gluing operator defined as follows:
$${\mathbf{G}}= \sum\limits_{\ard=0}^{\infty}\, z^{\ard}\,{\rm ev}_{p_1,*}\otimes{\rm ev}_{p_2,*}
\Big( \textsf{QM}^{\ard}_{{{\rm{relative}\, p_1}} \atop {{\rm{relative}\, p_2}}},
\widehat{{\O}}_{{\rm{vir}}}  \Big)\in K_{\mathsf{T}_q}^{\otimes 2}(X)[[z]].
$$
\end{theorem}

\noindent {\bf Remark.} Note, that along with the difference operators in equivariant parameters
there is an additional commuting set of difference operators for $\widehat{\Phi}$ in $z$-variables (see second equation in Theorem 8.2.20 of \cite{Okounkov:2015aa}), i.e. $${\widehat \Phi}(q^{\mathcal{L}} z)={\rm M}_{\mathcal{L}}{\widehat \Phi}(z),$$ which played a pivotal role in \cite{Pushkar:2016qvw, Koroteev:2017}.\\

Assume the action 
$\sigma : \mathbb{C}^{\times}\to A$  is such that 
$\mathbf{w}=a\mathbf{w}'+\mathbf{w}''$,
 $u\in \mathbb{C}^{\times}$. Then 
 \begin{equation}
 X^{\sigma}=\bigsqcup_{\mathbf{v}'+\mathbf{v}''=\mathbf{v}}N(\mathbf{v}',\mathbf{w}')\times N(\mathbf{v}'',\mathbf{w}'')\,. \notag
 \end{equation}

In section 9 of \cite{Okounkov:2015aa} the following maps in localized K-theory $\text{Stab}_{\pm}(a): K_T(X^{\sigma})\to K_T(X)$ were introduced, such that the R-matrix for affine Lie algebra $\widehat{\mathfrak{g}}$ associated to a given Nakajima variety is the composition of these maps\footnote{See \cite{2012arXiv1211.1287M} for the cohomological version.} 
\begin{equation}
R(a)=\text{Stab}_{-}^{-1}\text{Stab}_{+}\,.
\end{equation}

In the case of $A_n$ quiver varieties, $\widehat{\mathfrak{g}}=\widehat{\mathfrak{gl}}(n)$ and the R-matrix is the product of trigonometric R-matrices (as it is proven in \cite{Okounkov:2016sya}), associated with $\widehat{\mathfrak{gl}}(n)$ which we will discuss in detail in \secref{Sec:qKZtRS}. 

Moreover, the following theorem (see  Theorem 9.3.1 from  \cite{Okounkov:2015aa} as well as Theorem 3.1 of \cite{Okounkov:2016sya}) is true
\begin{theorem}\label{Th:ShiftOperator}
The shift operator $\mathbf S_{\sigma}$ is related to the R-matrix 
$R(a)$ in the following way
$$
\tau_\sigma^{-1}z^{{\mathbf v}'}R(a)=\text{Stab}_{+}^{-1}\tau_{\sigma}^{-1}\mathbf S_{\sigma}(a,z)\text{Stab}_{+},
$$
where $\tau_{\sigma}f(a)=f(q^{\sigma} a)$, identifying the difference equation from Theorem \ref{Th:qkzgeom} with the quantum Knizhnik-Zamolodchikov equation.
\end{theorem}
\noindent {\bf Remark.} We note, that here we use a modified $z$-variables as in \cite{Okounkov:2016sya}, namely, we should use the shifted variable  $z^{\bf v'}\to z^{\bf v'}(-1)^{\rm codim/2}$, i.e. in components, $z_i\to (-1)^{2\kappa_i}z_i$, so that $2\kappa_i=\mathbf{w}'-C\mathbf{v}'$, where $C$ is the Cartan matrix of quiver. 
 
 \subsection{Relative vs Descendant Insertions}
Finally, the last important topic we want to discuss is how the matrix elements of the fusion matrix can be expressed using vertices with descendants. This problem was recently solved by Aganagic and Okounkov \cite{Aganagic:2017be}.
Namely, for every element $\alpha \in K_T(X)$ which serves as a relative insertion at the vertex, one constructs a localized descendent insertion, i.e. an element $f_{\alpha}\in \mathbb{Q}(T\times T_{G}/W_G)$, where $T_G$ and $W_G$ are correspondingly maximal torus and Weyl group of $G=\prod_i GL(V_i)$. While being a rational function, $f_{\alpha}$ admits only specific kinds of denominators: 
$$
f_{\alpha}=\frac{s_{\alpha}}{\Delta_{\hbar}},
$$
such that $s_{\alpha}\in \mathbb{Z}(T\times T_{G}/W_G)$ and 
$$
\Delta_{\hbar}=\sum_k(-\hbar)^{-k}\Lambda^k \text{Lie}(G)=\prod_i\prod_{k,l}\Big (1-\frac{x_{i,k}}{\hbar x_{i,l}}\Big)\,,
$$
so that $x_{i,r}$ are grouped accordingly to $G=\prod_{i\in I} GL(V_i)$.

Therefore the following theorem of \cite{Aganagic:2017be} holds.

\begin{theorem}
The matrix elements of $\Psi$ can be expressed as follows:
$\Psi_{\bf p,\bf q}=V_{\bf p}^{(f_{\bf q})}$, where $\bf p$ stands for $\bf p$-th fixed point component of bare vertex with the descendant $f_{\bf q}\in \mathbb{Q}(T\times T_{G}/W_G)$ as described above.
\end{theorem}

\section{Vertex Functions}\label{sec:VertFunc}
In this section we describe vertex functions for partial flag varieties which were recently computed using localization technique in \cite{Koroteev:2017}. 

It will be convenient to introduce the following notation: $\mathbf{v}_i^{\prime}=\mathbf{v}_{i+1}-\mathbf{v}_{i-1}$, for $i=2,\ldots ,n-2$, $\mathbf{v}_{n-1}^{\prime}=\mathbf{w}_{n-1}-\mathbf{v}_{n-2}$, $\mathbf{v}_1^{\prime}=\mathbf{v}_{2}$. 

To describe the expression for the vertex one needs to take into account  the fixed points of $\textsf{QM}^{\bf d}_{{\rm nonsing} \, p_2}$. Each such a point is described by the data $(\{\mathscr{V}_i\},\{\mathscr{W}_{n-1}\})$, where ${\rm deg}\mathscr{V}_i=d_i, {\rm deg}\mathscr{W}_{n-1}=0$. Each bundle $\mathscr{V}_i$ can be decomposed into a sum of line bundles $\mathscr{V}_i=\mathcal{O} (d_{i,1}) \oplus \ldots \oplus \mathcal{O}(d_{i,{\bf v}_i})$ (here $d_i=d_{i,1}+\ldots +d_{i,{\bf v}_i}$). For a stable quasimap with such data to exist the collection of $d_{i,j}$ must satisfy the following conditions
\begin{itemize}
\item $d_{i,j}\geq 0$,
\item for each $i=1,\ldots ,n-2$ there should exist a subset in $\{d_{i+1,1},\ldots d_{i+1,{\bf v}_{i+1}}\}$ of cardinality ${\bf v}_i$ $\{d_{i+1,j_1},\ldots d_{i+1,j_{{\bf v}_i}}\}$, such that $d_{i,k}\geq d_{i+1, j_k}$.
\end{itemize}
In the following we will denote the chamber containing such collections $\{d_{i,j}\}$ as $\mathrm{C}$.

Now we are ready to write down the contributions for the entire vertex function:
\begin{theorem}\label{Prop:2017paper}
Let $\fp=\mathbf{V}_{1}\subset \ldots \subset \mathbf{V}_{n-1}\subset \{a_1,\cdots,a_{\mathbf{w}_{n-1}}\}$ $(\mathbf{V}_{i}=\{ x_{i,1},\ldots x_{i,\mathbf{v}_{i}}\})$ be a chain of subsets defining a torus fixed point $\fp\in X^{\mathsf{T}}$.
Then the coefficient of the vertex function for this point is given by:
$$
V^{(\tau)}_{\fp}(z) = \sum\limits_{d_{i,j}\in C}\, z^{\ard} q^{N(\ard)/2}\, EHG\ \ \tau(x_{i,j} q^{-d_{i,j}}),
$$
where $\ard=(d_1,\ldots ,d_{n-1}),d_i=\sum_{j=1}^{\mathbf{v}_i}d_{i,j},N(\ard)=\mathbf{v}_i^{\prime}d_i$,
\beq
E=\prod_{i=1}^{n-1}\prod\limits_{j,k=1}^{\mathbf{v}_i}\{x_{i,j}/x_{i,k}\}^{-1}_{d_{i,j}-d_{i,k}},\quad 
G=\prod\limits_{j=1}^{\mathbf{v}_{n-1}} \prod\limits_{k=1}^{\mathbf{w}_{n-1}} \{x_{n-1,j}/a_k\}_{d_{n-1,j}},\nonumber
\eeq
$$
H=\prod_{i=1}^{n-2}\prod_{j=1}^{\mathbf{v}_i}\prod_{k=1}^{\mathbf{v}_{i+1}}\{x_{i,j}/x_{i+1,k}\}_{d_{i,j}-d_{i+1,k}}.
$$
Here
\beq
\{x\}_{d}=\dfrac{(\hbar/x,q)_{d}}{(q/x,q)_{d}} \, (-q^{1/2} \hbar^{-1/2})^d, \ \ \textrm{where}  \ \ (x,q)_{d}=\prod^{d-1}_{i=0}(1-q^ix).\nonumber
\eeq
\end{theorem}

The same formula for the vertex can an be obtained using the following integral representation \cite{Aganagic:2017be,Aganagic:2017la}.

\begin{proposition}\label{Prop:VertexIntegralForm}
The $\bf p$-th component of a bare vertex function is given by
\beq \label{vertexint}
V^{(\tau)}_{\fp}(z)=
 \dfrac{1}{2 \pi i \alpha_{\fp}} \int\limits_{C_{\bf p}} \,\prod\limits_{i=1}^{n-1}\prod\limits_{j=1}^{\mathbf{v}_i} \dfrac{d s_{i,j}}{s_{i,j}} \, e^{-\frac{\log(z^{\sharp}_i) \log(s_{i,j}) }{\log(q)}} \, E_{\rm{int}} G_{\rm{int}} H_{\rm{int}} \tau(s_1,\cdots,s_k),
\eeq
where
$$
E_{\rm{int}}=\prod_{i=1}^{n-1}\prod\limits_{j,k=1}^{\mathbf{v}_i} \dfrac{\varphi\Big( \frac{s_{i,j}}{ s_{i,k}}\Big)}{\varphi\Big(\frac{q}{\hbar} \frac{s_{i,j}}{ s_{i,k}}\Big)}, \quad 
G_{\rm{int}}=\prod\limits_{j=1}^{\mathbf{w}_{n-1}}\prod\limits_{k=1}^{\mathbf{v}_{n-1}} \dfrac{\varphi\Big(\frac{q}{\hbar} \frac{s_{n-1,k}}{ a_j }\Big)}{\varphi\Big(\frac{s_{n-1,k}}{ a_j }\Big)},
$$
$$
H_{\rm{int}}=\prod_{i=1}^{n-2}\prod\limits_{j=1}^{\mathbf{v}_{i+1}}\prod\limits_{k=1}^{\mathbf{v}_{i}} \dfrac{\varphi\Big(\frac{q}{\hbar} \frac{s_{i,k}}{ s_{i+1,j} }\Big)}{\varphi\Big(\frac{s_{i,k}}{ s_{i+1,j} }\Big)},
$$
$$\alpha_{\fp}=\prod\limits_{i=1}^{n-1}\prod\limits_{j=1}^{\mathbf{v}_i}e^{-\frac{\log(z^{\sharp}_i) \log(s_{i,j}) }{\log(q)}} \, E_{\rm{int}} G_{\rm{int}} H_{\rm{int}}\Big|_{s_{i,j}=x_{i,j}},$$
where
\begin{equation}
\varphi(x)=\prod^{\infty}_{i=0}(1-q^ix)\,,
\label{eq:PhiDef}
\end{equation}
and the contour $C_{\bf p}$ runs around points corresponding to chamber $\mathrm{C}$ and the  shifted variable  $z^{\sharp}=z(-\hbar^{^1/_2})^{\det(\mathscr{P})}$. 
Here $z^{\sharp}=\prod_{i=1}^{n-1}z^{\sharp}_i$, 
so that $z^{\sharp}_i=z_i(-\hbar^{^1/_2})^{\mathbf{v}_i^{\prime}}$.

\end{proposition}
In \cite{Pushkar:2016qvw}, \cite{Koroteev:2017} we found these formulas to be useful to study their asymptotics at $q\to 1$ which lead to Bethe ansatz equations, producing the relations for the quantum K-theory ring. In this article, we however will leave parameter $q$ intact.

\section{Trigonometric RS Difference Operators}\label{Sec:tRS}
Proposition \ref{Prop:VertexIntegralForm} provides integral formulas for vertex functions $V_{\textbf{p}}$ of $X$ which depend on the choice of the contour $C_{\textbf{p}}$. In this section we study properties of integral \eqref{vertexint} without explicitly specifying the contour. 
In particular, we shall demonstrate that for a properly chosen contour \eqref{vertexint} solves quantum difference equations of the trigonometric Ruijsenaars-Schneider model. In this work we shall only study difference equations in equivariant parameters of $X$, see \cite{Koroteev:2018isw} (Theorem 2.6). 

In full generality tRS Hamiltonians read\footnote{In this section we use slightly different normalization of the tRS operators than in \cite{Koroteev:2017}.}
\begin{equation}
T_r(\textbf{a})=\sum_{\substack{\mathcal{I}\subset\{1,\dots,n\} \\ |\mathcal{I}|=r}}\prod_{\substack{i\in\mathcal{I} \\ j\notin\mathcal{I}}}\frac{t\,a_i-a_j}{a_i-a_j}\prod\limits_{i\in\mathcal{I}}p_i \,,
\label{eq:tRSRelationsEl}
\end{equation}
where $\textbf{a}=\{a_1,\dots, a_{\textbf{w}_{n-1}}\}$, the shift operator $p_i f(a_i)=f(qa_i)$ and we denoted $t=\frac{q}{\hbar}$.

In order to understand how the above difference operators act on integrals of the form \eqref{vertexint} we need to study in  detail how they act on the ingredients of the integrand. In what follows we shall describe these actions for vertex functions of quiver variety $X$ in question. The analysis for cotangent bundles to complete flag varieties was performed in  \cite{2012JMP....53l3512H} and in \cite{Bullimore:2015fr}.

Consider the following function 
\begin{equation}
H_{\textbf{v}_n,\textbf{v}_{n+1}}(\textbf{s}_{n},\textbf{s}_{n+1})=\prod\limits_{k=1}^{\textbf{v}_n}\prod\limits_{j=1}^{\textbf{v}_{n+1}}\dfrac{\varphi\left(\frac{q}{\hbar} \frac{s_{n,k}}{ s_{n+1,j} }\right)}{\varphi\left(\frac{s_{n,k}}{ s_{n+1,j} }\right)},
\label{eq:HyperGen}
\end{equation}
were $\textbf{s}_{n}=\{s_{n,1},\dots,s_{n,\textbf{v}_n}\}$ and $\textbf{s}_{n+1}=\{s_{n+1,1},\dots,s_{n+1,\textbf{v}_{n+1}}\}$. The following lemma describes action of the difference operator $p_{n,k}$
\begin{equation}
p_{n,k} f(s_{n,1},\dots, s_{n,k},\dots s_{n,\textbf{v}_n}) =  f(s_{n,1},\dots, q s_{n,k},\dots s_{n,\textbf{v}_n})\,.
\end{equation}
on this function.

\begin{lemma}\label{Th:Lemma1}
Let $H$ be given in \eqref{eq:HyperGen} then
\begin{equation}
p_{n,k} H_{\textbf{v}_n,\textbf{v}_{n+1}}(\textbf{s}_{n},\textbf{s}_{n+1}) = \prod\limits_{j=1}^{\textbf{v}_{n+1}}\frac{s_{n+1,j}-s_{n,k}}{s_{n+1,j}-\frac{q}{\hbar}s_{n,k}}\cdot H_{\textbf{v}_n,\textbf{v}_{n+1}}(\textbf{s}_{n},\textbf{s}_{n+1})\,.
\end{equation}
\label{Lemma:Actionofp}
\end{lemma}

\begin{proof}
From the definition of the function $\varphi(x)$ \eqref{eq:PhiDef} we get the identity $p_{n,k} \varphi(s_{n,k})=(1-s_{n,k})^{-1} \varphi(s_{n,k})$, which we need to apply twice.
\end{proof}

Note another useful identity 
\begin{equation}
p^{-1}_{n,k} \varphi(s_{n,k}):= \varphi(q^{-1} s_{n,k})=(1-q^{-1} s_{n,k}) \varphi(s_{n,k})\,.
\label{eq:PinvId}
\end{equation}
Let us also define tRS operators with the opposite shift 
\begin{equation}
T'_r(\textbf{a})=\sum_{\substack{\mathcal{I}\subset\{1,\dots,n\} \\ |\mathcal{I}|=r}}\prod_{\substack{i\in\mathcal{I} \\ j\notin\mathcal{I}}}\frac{t\,a_i-a_j}{a_i-a_j}\prod\limits_{i\in\mathcal{I}}p^{-1}_i \,,
\label{eq:Trnpr}
\end{equation}
where $p^{-1}_i f(a_i)=f(q^{-1}a_i)$.

Next, we need to prove another lemma which directly follows from Lemma \ref{Lemma:Actionofp}.
\begin{lemma}\label{Th:Lemma2}
Consider function \eqref{eq:HyperGen} with $\textbf{v}_{n+1}=\textbf{v}_{n}$. Then
\begin{equation}
T_r(\textbf{s}_n) H_{\textbf{v}_n,\textbf{v}_{n}}\left(\textbf{s}_{n},\textbf{s}_{n+1}\right) = T'_r\left(\textbf{s}_{n+1}\right) H_{\textbf{v}_n,\textbf{v}_{n}}(\textbf{s}_{n},\textbf{s}_{n+1})\,.
\label{eq:TrHvv}
\end{equation}
\end{lemma}

\begin{proof}
The lemma follows from the direct computation after using the identity which was proven in \cite{2012JMP....53l3512H}
\begin{align}
&\sum_{\substack{\mathcal{I}\subset\{1,\dots,\textbf{v}_n\} \\ |\mathcal{I}|=r}}\prod_{\substack{i\in\mathcal{I} \\ j\notin\mathcal{I}}}\frac{t\, s_{n,i}-s_{n,j}}{s_{n,i}-s_{n,j}}\cdot  \prod\limits_{j=1}^{\textbf{v}_{n}}\frac{s_{n+1,i}-s_{n,j}}{s_{n+1,i}-t\,s_{n,j}} \cr
&= 
\sum_{\substack{\mathcal{I}\subset\{1,\dots,\textbf{v}_n\} \\ |\mathcal{I}|=r}}\prod_{\substack{i\in\mathcal{I} \\ j\notin\mathcal{I}}}\frac{t\, s_{n+1,i}-s_{n+1,j}}{s_{n+1,i}-s_{n+1,j}}\cdot  \prod\limits_{j=1}^{\textbf{v}_{n}}\frac{s_{n+1,i}-s_{n,j}}{s_{n+1,i}-t\,s_{n,j}}\,.
\label{eq:IdNN}
\end{align}
\end{proof}

Using the above lemma we can find generalizations of \eqref{eq:TrHvv} to the cases when tRS operators act on function \eqref{eq:HyperGen} with different labels $\textbf{v}_n$ and $\textbf{v}_{n+1}$. The action can be expressed in terms of quantum dimensions of the irreducible representations of $\mathfrak{gl}_{\textbf{v}_{n+1}-\textbf{v}_n}$.

\begin{definition}
Let $\Lambda_s^N$ be the $s$-th antisymmetric tensor power of the fundamental representation of $\mathfrak{gl}_{N}$. Then its quantum dimension is given by
\begin{equation}
\text{qdim}(\Lambda_s^{N}) = s_{\Lambda_s^{N}}(t^{N-1},\dots, t^{1-N})\,,
\end{equation}
where $t$ is the quantum parameter and $s_\rho(x_1,\dots,x_{2N-1})$ is symmetric Schur polynomial for partition $\rho$.
\end{definition}
Note that in the above definition the quantum dimensions of $\Lambda_s^{N}$ are nothing but elementary symmetric  polynomial $\text{qdim}(\Lambda_s^{N})=e_s(t^{N-1},\dots, t^{1-N})$ which satisfies 
\begin{equation}
e_s(t^{N},\dots, t^{-N})=t^s e_s(t^{N-1},\dots, t^{1-N})+t^{s-1-N}e_{s-1}(t^{N-1},\dots, t^{1-N})\,.
\label{eq:QDimident}
\end{equation}

We can now formulate
\begin{lemma}\label{Th:Lemma3}
Assume that $\textbf{v}_{n+1}-\textbf{v}_n\geq 0$. Then
\begin{equation}
T_r(\textbf{s}_n)\cdot H_{\textbf{v}_n,\textbf{v}_{n+1}}\left(\textbf{s}_{n},\textbf{s}_{n+1}\right) = 
\sum_{s=0}^{\text{min}(r,\textbf{v}_{n+1}-\textbf{v}_n)} \text{qdim}(\Lambda_s^{\textbf{v}_{n+1}-\textbf{v}_n})\,T'_{r-s}\left(\textbf{s}_{n+1}\right)\cdot H_{\textbf{v}_n,\textbf{v}_{n+1}}(\textbf{s}_{n},\textbf{s}_{n+1})\,,
\label{eq:TrHvvpM}
\end{equation}
where it is assumed that $T'_0(\textbf{s}_{n+1})=1$.
\label{Lemma:QuantumDim}
\end{lemma}

\begin{proof}
We start with \eqref{eq:IdNN} where we replace $\textbf{v}_{n}$ by $\textbf{v}_{n+1}$. Then we take the limit 
$s_{n+1,j}\to\infty$, where $j= \textbf{v}_{n+1}-\textbf{v}_{n}, \textbf{v}_{n+1}-\textbf{v}_{n}+1,\dots, \textbf{v}_{n+1}$\,.
Analogously to the previous lemma the desired identity follows after acting with the shift operators and applying \eqref{eq:QDimident}.
\end{proof}

For the first tRS Hamiltonian $T_1$ identity \eqref{eq:TrHvvpM} simplifies as follows.
\begin{corollary}\label{Th:Lemma4}
\begin{equation}
T_1(\textbf{s}_n)\cdot H_{\textbf{v}_n,\textbf{v}_{n+1}}\left(\textbf{s}_{n},\textbf{s}_{n+1}\right) = 
\left[T'_1\left(\textbf{s}_{n+1}\right)+\frac{t^{\textbf{v}_{n+1}-\textbf{v}_{n}}-t^{-\textbf{v}_{n+1}+\textbf{v}_{n}}}{t-t^{-1}}\right]\cdot H_{\textbf{v}_n,\textbf{v}_{n+1}}(\textbf{s}_{n},\textbf{s}_{n+1})\,.
\label{eq:TrHvvpMr1}
\end{equation}
\end{corollary}
Note that \eqref{eq:TrHvvpM} can be proven using this corollary by induction using \eqref{eq:QDimident}. Equation \eqref{eq:TrHvvpMr1} relates first tRS Hamiltonian with its conjugate of its decomposition into representations of $U(N-1)$  to $\Lambda^N_r$ into $\Lambda^{N-1}_r\oplus \Lambda^{N-1}_{r-1}$.

Let us now study the action of tRS operators on function $E$. The following lemma can be easily proven using properties of function $\varphi$.
\begin{lemma}\label{Th:Lemma5}
Let $E$ be defined as
\begin{equation}
E(\textbf{s}_n)=\prod\limits_{j,k=1}^{\mathbf{v}_n} \dfrac{\varphi\Big( \frac{s_{n,j}}{ s_{n,k}}\Big)}{\varphi\Big(t \frac{s_{n,j}}{ s_{n,k}}\Big)}\,,
\end{equation}
then it satisfies the following difference relation for the inverse shift
\begin{equation}
p_{n,k}^{-1} E (\textbf{s}_n) = \prod\limits_{i,k=1}^{\mathbf{v}_n}  \dfrac{q^{-1}s_{n,i}-s_{n,k}}{q^{-1}ts_{n,i}-s_{n,k}} \dfrac{t\, s_{n,i}-s_{n,k}}{s_{n,i}-s_{n,k}} \cdot E (\textbf{s}_n)\,.
\label{eq:shiftEinverse}
\end{equation}
\end{lemma}

The important property of the tRS difference operators is that they are self-adjoint with respect to the measure $\frac{d\textbf{s}_n}{\textbf{s}_n} \cdot E (\textbf{s}_n):=\prod\limits_{i=1}^{\textbf{v}_n} \frac{ds_{n,i}}{s_{n,i}}E (s_{n,i})$ on the Cartan subalgebra of $U(\mathbf{v}_n)$. We shall prove this property in the following lemma.

\begin{lemma}\label{Lemma:ShiftLemma}
Let $f$ and $g$ be meromorphic functions of their arguments. Then, provided that contour $C$ does not encounter any poles of these functions upon shift $s_{n,k}\to q^{-1} s_{n,k}$, the following identity holds
\begin{equation}
\int\limits_C \frac{d\textbf{s}_n}{\textbf{s}_n} \cdot E (\textbf{s}_n)f(\textbf{s}_n)\, \left[T_r (\textbf{s}_n) \cdot g(\textbf{s}_n)\right] = \int\limits_C \frac{d\textbf{s}_n}{\textbf{s}_n} \cdot E (\textbf{s}_n) \left[T'_r (\textbf{s}_n) \cdot f(\textbf{s}_n)\right] g(\textbf{s}_n)\,.
\label{Lemma:Intbyparts}
\end{equation}
\end{lemma}

\begin{proof}
Consider $s_{n,k}$ where $k\in\mathcal{I}$ from the definition of difference operators $T_r$ \eqref{eq:tRSRelationsEl}.
Assuming that we do not hit any poles, we shift the contour of integration by $s_{n,k}\to q^{-1} s_{n,k}$ only for $k\in\mathcal{I}$. This operation can be expressed via acting with the inverse shift $p_{n,k}^{-1}$ on the integrand of the left hand side of \eqref{Lemma:Intbyparts}
\begin{equation}
\int\limits_C \prod\limits_{i=1}^{\textbf{v}_n} \frac{ds_{n,i}}{s_{n,i}}  \sum_{\substack{\mathcal{I}\subset\{1,\dots,n\} \\ |\mathcal{I}|=r}}\left[\prod_{k\in\mathcal{I}} p_{n,k}^{-1}  \cdot E(s_{n,i}) f(\textbf{s}_n)\right] 
\cdot \prod_{\substack{i\in\mathcal{I} \\ j\notin\mathcal{I}}}\frac{tq^{-1}\,s_{n,i}-s_{n,j}}{q^{-1}s_{n,i}-s_{n,j}}\cdot g(\textbf{s}_n)\,.
\end{equation}
Using \eqref{eq:PinvId} and \eqref{eq:shiftEinverse} we arrive to the right hand side of \eqref{Lemma:Intbyparts}.
\end{proof}

\subsection{tRS Difference Equations}
Now we shall use the lemmas which we have just proven to construct a solution for the quantum difference tRS equations.
First, let us change quantum parameters in K-theory as follows
\begin{align}
z^\sharp_1&=\frac{\zeta_{1}}{\zeta_2}\,,\cr
z_i^\sharp&=\frac{\zeta_{i}}{\zeta_{i+1}}\,,\quad i=2,\dots, n-2\cr
z_{n-1}^\sharp&=\frac{\zeta_{n-1}}{\zeta_{n}}\,.
\end{align}

\begin{theorem}\label{Th:tRSEigen}
The following function constructed for the cotangent bundle to the partial flag variety $X$ labelled by $\textbf{v}_{1},\dots,\textbf{v}_{n-1},\textbf{w}_{n-1}$
\begin{equation}
\mathrm{V}(\textbf{a},\vec{\zeta}) = \frac{e^{\frac{\log \zeta_n \sum_{i=1}^{n-1}\log a_i}{\log q}}}{2\pi i} \int\limits_{C} \prod_{m=1}^{n-1}\prod_{i=1}^{\textbf{v}_m} \frac{ds_{m,i}}{s_{m,i}} E(s_{m,i})\,\, e^{-\frac{\log \zeta_{m}/\zeta_{m+1} \cdot\log s_{m,i}}{\log q}} \cdot \prod_{j=1}^{\textbf{v}_{m+1}}H_{\textbf{v}_m,\textbf{v}_{m+1}}\left(s_{m,i},s_{m+1,j}\right)\,,
\label{eq:GenerictRSSolution}
 \end{equation}
where contour $C$ is chosen in such a way that shifts of the contour $\textbf{s}\to q^{\pm 1} \textbf{s}$ do not encounter any poles, 
satisfies tRS difference relations
\begin{equation}
T_r(\textbf{a}) \mathrm{V}(\textbf{a},\vec{\zeta}) = S_r (\vec{\zeta},t) \mathrm{V}(\textbf{a},\vec{\zeta})\,, \qquad r=1,\dots, {\bf w_{n-1}}
\label{eq:tRSEigen}
\end{equation}
where function $S_r $ is $r$-symmetric polynomial of the following $\sum_{k=1}^n k s_k $ variables
\begin{equation}
\{t^{\textbf{v}'_1-1}\zeta_1,\dots, t^{-\textbf{v}'_1+1}\zeta_1,\dots\dots,t^{\textbf{v}'_{n-1}-1}\zeta_n,\dots, t^{-\textbf{v}'_{n-1}+1}\zeta_n\}\,,
\label{eq:symvariables}
\end{equation}
where $\textbf{v}'_{i}=\textbf{v}_{i+1}-\textbf{v}_{i}$ for $i=1,\dots {n-2}$ and $\textbf{v}'_{n-1}=\textbf{w}_{n-1}-\textbf{v}_{n-1}$.
\end{theorem}

\begin{proof}
First, we need to justify that contour $C$ can be always chosen in such a way that its shifts do not result in any additional residues and that Lemma \ref{Lemma:ShiftLemma} can be applied. Indeed, in any given complex plane $s_{n,i}$ poles in the integrand of \eqref{eq:GenerictRSSolution} are located at $s_{n,i}=\sigma q^{-d_{n,i}}$ for some $\sigma$ (different for each plane). The contour can be safely chosen to avoid the collision with poles. Indeed, suppose $q$ is real so various poles of the integrand are located on lines which are parallel to the real axis. The contour is therefore chosen to go above and below the above string of poles. Since the contour is parallel to the real axis it won't be affected by the $q$-shift.

We prove the theorem by induction. The base of induction is $A_1$ quiver with no $\mathcal{V}_i$ bundles, but merely constant bundle $\mathcal{W}_1$ or rank ${\bf v_1}$ (in other words, Grassmannian $Gr_{0,{\bf v_1}}$).
For this quiver the function $V_p$ has no integration
\begin{equation}
\mathrm{V} (\textbf{s}_1,\zeta_1) = e^{\frac{\log \zeta_1 \sum_{i=1}^{{\bf v_1}}\log s_{1,i}}{\log q}}\,,
\end{equation}
which satisfies
\begin{equation}
T_r(\textbf{s}_1) \mathrm{V} (\textbf{s}_1,\zeta_1) = e_r\left(t^{\frac{1-{\bf v_1}}{2}},\dots,t^{\frac{{\bf v_1}-1}{2}}\right)\zeta_1^r\cdot  \mathrm{V} (\textbf{s}_1,\zeta_1)\,,
\end{equation}
where $e_r$ is elementary $r$-th symmetric polynomial of its arguments. This follows directly from the structure of tRS operators.

Now let us assume that we for the quiver variety $X$ with $n-2$ nodes labelled by ${\bf v_1},\dots\textbf{v}_{n-2}$ relation \eqref{eq:tRSEigen} holds. Let us add another node of rank $\textbf{v}_{n-1}$. Using this decomposition we can rewrite the following function
\begin{equation}
\mathrm{V}(\textbf{a},\vec{\zeta}^{(n)}) = e^{\frac{\log \zeta_n \sum_{i=1}^{\textbf{w}_{n-1}}\log a_i}{\log q}}\,  \int\limits_{C^{(n-1)}} \prod_{i=1}^{\textbf{v}_{n-1}} d E(s_{n-1,i})\, H_{\textbf{v}_{n-1},\textbf{w}_{n-1}}\left(s_{n-1,i},\textbf{a}\right)\cdot  \mathrm{V}(\textbf{s}_{n-1},\vec{\zeta}^{(n-1)})\,,
\end{equation}
for the proper choice of the contour $C^{(n-1)}$ and where we indicate different numbers of parameters $\zeta$ in the arguments of the vertex functions by the corresponding superscripts.

By acting on the integral with the tRS operator $T_r(\textbf{s}_{n})$ (here we identify equivariant parameters $\textbf{a}$ with Bethe roots $\textbf{s}_{n}$) we get the following expression
\begin{align}
&T_r(\textbf{a}) \mathrm{V}(\textbf{a},\vec{\zeta}^{(n)})= e_r\left(t^{\frac{1-\textbf{w}_{n-1}}{2}},\dots,t^{\frac{\textbf{w}_{n-1}-1}{2}}\right) \zeta_n^r \cdot e^{\frac{\log \zeta_n \sum_{i=1}^{\textbf{w}_{n-1}}\log a_i}{\log q}} \cr
&\int\limits_{C^{(n-1)}} \prod_{i=1}^{\textbf{v}_{n-1}} d E(s_{n-1,i})\,\left[T_r(\textbf{a}) \cdot H_{\textbf{v}_{n-1},\textbf{w}_{n-1}}(s_{n-1,i},\textbf{a})\right] \cdot  \mathrm{V}(\textbf{s}_{n-1},\vec{\zeta}^{(n-1)}/\zeta_n)\,.
\end{align}

Using Lemma \ref{Lemma:QuantumDim} we can replace the action of the operator $T_r(\textbf{a})$ on variables $a_j$ by the right hand side of \eqref{eq:TrHvvpM}. Then, by employing Lemma \ref{Lemma:ShiftLemma} we can `integrate by parts' each term in the resulting sum. Finally we use the inductive assumption about the eigenvalues $S_{r-s} (\vec{\zeta}^{(n-1)}/\zeta_n,t)$ of $T_{r-s}(\textbf{s}_{n-1})$ to get the following formula for the eigenvalue of $T_r(\textbf{a})$:
\begin{equation}
S_r\left(\zeta_n t^{\frac{1-\textbf{w}_{n-1}}{2}},\dots,\zeta_n t^{\frac{\textbf{w}_{n-1}-1}{2}}\right) \cdot \sum_{s=0}^{\text{min}(r,\textbf{v}_{n}')} \text{qdim}\left(\Lambda_s^{\textbf{v}_{n}'}\right)\cdot S_{r-s} (\vec{\zeta}^{(n-1)}/\zeta_n,t)\,,
\label{eq:Eigenfinal}
\end{equation}
where for the first polynomial we used that 
$S_r\left(\zeta_n t^{\frac{1-\textbf{w}_{n-1}}{2}},\dots,\zeta_n t^{\frac{\textbf{w}_{n-1}-1}{2}}\right)=\\ \zeta_n^r e_r\left(t^{\frac{1-\textbf{w}_{n-1}}{2}},\dots,t^{\frac{\textbf{w}_{n-1}-1}{2}}\right)\,.$ It can be shown using properties of symmetric polynomials that the above expression is equal to an $r$-symmetric polynomial of variables listed in \eqref{eq:symvariables}. We shall illustrate this fact for the case when $\textbf{v}_n'=1$, then the expression in \eqref{eq:Eigenfinal} reads
\begin{equation}
\zeta_n^r \left(S_{r} (\vec{\zeta}^{(n-1)}/\zeta_n,t) +S_{r-1}(\vec{\zeta}^{(n-1)}/\zeta_n,t)\right)=S_{r}(\zeta^{(n)})\,,
\label{eq:Eigenfinal1}
\end{equation}
where we used \eqref{eq:QDimident} at the last step. The analogue of the above equality for higher $\textbf{v}_n'$ can be shown by induction.
\end{proof}

\subsection{Vertex Functions from tRS Eigenfunctions}
We can now demonstrate how to compute vertex functions from Proposition \ref{Prop:2017paper} using the general tRS solution \eqref{eq:GenerictRSSolution} by properly specifying the contour of integration. In order to do that we need to understand how to identify each chamber $\mathrm{C}$ by choosing contour $C$ in \eqref{eq:GenerictRSSolution}. 

The prescription goes as follows. We shall only pick poles of functions $H$ in the integrand and ignore poles of $E$ functions. It can be argued (see Section 3 of \cite{Bullimore:2015fr}) that for a contour which encircles \textit{all} poles of $\phi$ functions of the integrand only poles of $H$ functions survive, whereas poles of $E$ functions are cancelled by zeroes of $H$'s. Poles of $H_{\textbf{v}_n,\textbf{v}_{n+1}}(\textbf{s}_{n},\textbf{s}_{n+1})$ have the form $s_{n,i}/s_{n+1,k}= q^{-d_{n,i}}$ for some nonnegative degrees $d_{n,i}$. The contour can then be chosen to go only around the poles whose degrees $d_{n,i}$ satisfy the inequalities in the definition of chambers described in the beginning of \secref{sec:VertFunc}.

\vspace{.3cm}
Once the integration contour $C$ is properly chosen for a given fixed point $\textbf{p}$ the integral from \ref{Prop:2017paper} can be readily evaluated.
\begin{theorem}\label{trsvertex}
Consider $\alpha_{\textbf{p}}$ and $V^{(1)}_{\textbf{p}}$ as defined previously in Theorem \ref{Prop:2017paper}. Then for each fixed point $\textbf{p}$ of the maximal torus of $X$ there is a contour $C$ for which integral \eqref{eq:GenerictRSSolution} evaluates to
\begin{equation}
\mathrm{V} = e^{\frac{\log \zeta_n \sum_{i=1}^{n-1}\log a_i}{\log q}} \alpha_{\textbf{p}} V^{(1)}_{\textbf{p}}\,.
\label{eq:VVpcorr}
\end{equation}
\end{theorem}

Let us illustrate this statement on a simple example.

\subsection{Example for $T^*\mathbb{P}^1$}
The vertex function \eqref{vertexint} for $T^*\mathbb{P}^1$ for a trivial class $\tau=1$ reads
\begin{equation}
V^{(1)}_{\textbf{p}}=\frac{1}{2\pi i \alpha_p}\int\limits_{C_{\textbf{p}}} \frac{ds}{s}\, (z^{\sharp})^{-\frac{\log s}{\log q}}\,\prod_{i=1}^2
\frac{\varphi\left(t\frac{s}{a_i}\right)}{\varphi\left(\frac{s}{a_i}\right)}\,,
\end{equation}
for the two fixed points $\textbf{p}=\{a_1\}$ and $\textbf{p}=\{a_2\}$. The poles are given by $s = a_p q^{-d}$ for nonnegative $d$. By taking the residues we arrive to the q-hypergeometric function
\begin{equation}
V^{(1)}_{\textbf{p}}=\sum_{d>0} (z^{\sharp})^d\,\prod_{i=1}^2
\frac{\left(\frac{q}{\hbar}\frac{a_\textbf{p}}{a_i};q\right)_d}{\left(\frac{a_\textbf{p}}{a_i};q\right)_d}= _2\!\!\phi_1\left(t,t\frac{a_{\textbf{p}}}{a_{\bar{\textbf{p}}}},\frac{a_{\textbf{p}}}{a_{\bar{\textbf{p}}}};q;z^{\sharp}\right)\,.
\label{eq:TP1V}
\end{equation}
Here $a_{\bar{\textbf{1}}}=a_{\textbf{2}}$ and $a_{\bar{\textbf{2}}}=a_{\textbf{1}}$. One can easily see that Weyl reflection of $\mathfrak{gl}_2$ interchanges $V^{(1)}_{\textbf{1}}$ and $V^{(1)}_{\textbf{2}}$.

The integral in \eqref{eq:GenerictRSSolution} for $T^*\mathbb{P}^1$ reads as follows
\begin{equation}
\mathrm{V}=\frac{e^{\frac{\log \zeta_2 \log a_1 a_2}{\log q}}}{2\pi i}\int\limits_{C} \frac{ds}{s}\, \left(\frac{\zeta_1}{\zeta_2}\right)^{-\frac{\log s}{\log q}}\,\prod_{i=1}^2
\frac{\varphi\left(\frac{q}{\hbar}\frac{s}{a_i}\right)}{\varphi\left(\frac{s}{a_i}\right)}\,.
\end{equation}
The denominator has two semi-infinite strings of poles at $s=a_i q^{-d}$ for $i=1,2$ and $d\in\mathbb{Z}_{\geq 0}$. In order to reproduce vertex functions we can pick the contour such that it encloses only the poles which start at $a_1$ or the one which starts at $a_2$ contour. After a straightforward calculation we find that
\begin{equation}
\mathrm{V} = e^{\frac{\log \zeta_2 \log a_1 a_2}{\log q}} \alpha_{\textbf{p}} V^{(1)}_{\textbf{p}}\,
\end{equation}
for each fixed point.

\vspace{3mm}
\noindent\textbf{Remark.} Theorem \ref{Th:tRSEigen} proves that the K-theory vertex function of the cotangent bundle to the partial flag variety in type A is an eigenfunction of the tRS difference operator. Using similar techniques, in particular Lemmas \ref{Th:Lemma1} -- \ref{Th:Lemma5}, with slight modifications, it will be possible to demonstrate that K-theory vertex functions for a more generic quiver varieties of type A (so that each node of the quiver may have framing) describe the spectrum of a certain restrictions of the quantum tRS integrable systems (the construction involves nilpotent orbits of $\mathfrak{gl}_N$ and Slodowy slices, see \cite{Gaiotto:2013bwa} Sections 2.7, 4). Some calculations in this direction were done in the physics literature (see \cite{Bullimore:2015fr}, Section 3). 
\subsection{Quantum Toda Chain from Trigonometric RS Model}
The construction described in this section provides explicit construction for the spectrum of the tRS model in terms of integrals of the Mellin-Barnes type. Prior to our construction similar ideas were developed for difference Toda chains \cite{Kharchev:2002, Gerasimov:2008ao}, see also geometric approach of  \cite{Braverman:2005kq,Braverman:2011si,2014arXiv1410.2365B}. Here we would like to comment on how to obtain the results in \textit{loc. cit.} from our construction. 

It was shown in \cite{Koroteev:2017} that the integrals of motion of classical tRS integrable system played a role of relations in the quantum equivariant K-theory ring of the cotangent bundle to (complete) flag variety. It was also proven in the same paper that in the $\hbar\to\infty$ limit the tRS integrals of motion become the q-Toda integrals of motion. Thus, the version of the quantum equivariant K-theory of flag varieties emerges in this limit, which magically repeats the results of \cite{2001math8105G}, where a different approach to K-theoretic counts was used.

Based on the computations of the current manuscript we can make a quantum analogue of that statement\footnote{Word `quantum' is used here in two different contexts -- quantum K-theory in the previous paragraph and quantum integrable system.}. Indeed, \eqref{eq:tRSEigen} are quantum (read q-difference) equations of the tRS model and function $\textrm{V}$ is their formal solution, which can be used to derive K-theoretic vertex functions for $X$ \eqref{eq:VVpcorr}. Therefore we expect that in the limit $\hbar\to\infty$ K-theoretic vertex functions for the cotangent bundle to the complete flag variety $T^*G/P$ will transform into Givental $J$-functions for the corresponding flags $G/P$\footnote{Here $P$ is a parabolic subgroup of $G$. In this paper $G=U(\textbf{w}_{n-1})$ and $P=U(\textbf{w}_{n-1}-\textbf{v}_{n-1})\times\dots\times U(\textbf{v}_{2}-\textbf{v}_{1})\times U(\textbf{v}_{1})$}. A similar statement is expected in cohomology, however, in K-theory it demonstrates a non-trivial connection between the theory of quasimaps which we used to study $T^*G/P$ and theory of stable maps which was used by Givental et al. 

As an illustration let us take $X=T^*\mathbb{P}^1$ again. Taking the limit $\hbar\to\infty$\footnote{The so-called Inosemtsev limit \cite{inozemtsev1989} in the literature on integrable systems (see more on this in \cite{GKKV:2019}).} (or $t\to 0$) in \eqref{eq:TP1V} we obtain
\begin{equation}
V^{(1)}_{\textbf{p}} \to _2\!\!\phi_1\left(0,0,\frac{a_{\textbf{p}}}{a_{\bar{\textbf{p}}}};q;z^{\sharp}\right)=:_1\!\!\phi_0\left(\frac{a_{\textbf{p}}}{a_{\bar{\textbf{p}}}};q;z^{\sharp}\right)=\sum\limits_{k=0}^{\infty}\frac{(z^{\sharp})^k}{\left(\frac{a_{\textbf{p}}}{a_{\bar{\textbf{p}}}},q\right)_k(q,q)_k}\,,
\end{equation}
where $\bf p$ and $\bf \bar{p}$ denote two fixed points, which up to a constant coincides with the Givental $J$-function for $\mathbb{P}^1$ from \cite{2001math8105G}. 

We would like to emphasize that, despite the coincidences above, Givental's and Okounkov's approaches to quantum K-theoretic counts are conceptually different and more works needs to be done in order to understand the exact relationship.

\section{qKZ versus tRS}\label{Sec:qKZtRS}
In this final section we shall study in detail qKZ equations for the cotangent bundle to partial flag variety and their solutions. After reminding the reader about the exact correspondences between K-theoretic and representation-theoretic data, we shall discuss the derivation of the relationship between the qKZ equations and the tRS eigenvalue problem along the lines of the paper by Zabrodin and Zotov \cite{Zabrodin:2017}. We shall provide the generalization of their argument in application to the trigonometric RS model (The \textit{loc. cit.} offers a proof only for the rational case). In the end we show that this analysis leads to a nontrivial relation for K-theoretic vertex functions with relative insertions, which is the second main result of this paper.

\subsection{Notations and Conventions} 
In \secref{Sec:QKZNak} we described Nakajima quiver varieties $N(\textbf{v},\textbf{w})$ of type $A$. Let us now consider the union of the cotangent bundles for partial flag varieties $\bigsqcup_{\mathbf{v}} N(\textbf{v},\textbf{w})$ with fixed framing. Its localized quantum K-theory, as a vector space, is spanned by K-theory classes corresponding to fixed points of the maximal torus. The corresponding vector spaces can be identified with the standard weight subspaces in 
\begin{equation}
{\rm V}=V(a_1)\otimes\dots\otimes V(a_{\bf w_{n-1}})=\oplus_{\{s_a\}} {\rm V}(\{s_a\}),\label{eq:WeightDecomp}
\end{equation}
where $V(a)$ is an $n$-dimensional evaluation representation of $U_{\hbar}(\widehat{\mathfrak{gl}}_n)$ where $a$
plays a role of the evaluation parameter \cite{Chari:1994pz} and the weight parameters $s_a$ are the eigenvalues of $S_a=\sum^n_{k=1} e^{(k)}_{aa}$. 

The identification of the weight subspaces ${\rm V}(\{s_a\})$ with K-theoretic data is as follows:
\begin{equation}
{\rm V}(\{s_a\})=K^{\rm loc}_T(N(\mathbf{v}, \mathbf{w}))\,,
\end{equation}
provided that $s_a=\mathbf{v}_{a}-\mathbf{v}_{a-1}$, where  $\mathbf{v}_{n}\equiv \mathbf{w}_{n-1}$ and $s_1=\mathbf{v}_{1}$. 

Let $R$ be a trigonometric R-matrix:
\begin{eqnarray}
&&R(x): V(xy)\otimes V(y)\to V(xy)\otimes V(y)\,, \nonumber\\
&&R(x)=\sum^{n}_{a=1}e_{aa}\otimes e_{aa}+
\frac{x-x^{-1}}{\hbar^{\frac{1}{2}} x-\hbar^{-\frac{1}{2}}x^{-1}}\sum^n_{a\neq b}e_{aa}\otimes e_{bb}+\nonumber\\
&&\frac{\hbar^{\frac{1}{2}}-\hbar^{-\frac{1}{2}}}{\hbar^{\frac{1}{2}} x-\hbar^{-\frac{1}{2}}x^{-1}}\sum_{a< b} (x e_{ab}\otimes e_{ba}-x^{-1}e_{ba}\otimes e_{ab}).
\end{eqnarray}
Let $R_{ij}(x)$ be the notation for the R-matrix operator acting in the $i$-th and $j$-th cites of $\mathcal{V}$. 
Then let us define qKZ operator as follows:
\begin{eqnarray}
&&K_i^{(q)}: {\rm V}\to {\rm V}\,,\nonumber\\
&&K_i^{(q)}=R_{i\, i-1}\left(\frac{a_i q}{a_{i-1}}\right)\dots R_{i\, 1}\left(\frac{a_iq}{a_1}\right)Z^{(i)}R_{i\, {\bf w_{n-1}}}\left(\frac{a_i}{a_{\bf w_{n-1}}}\right)\dots R_{i\, i+1}\left(\frac{a_i}{a_{i+1}}\right) ,
\end{eqnarray}
where $Z^{(i)}$ is an auxiliary diagonal matrix $\text{diag}(\zeta_1, \dots, \zeta_n)$ acting on $i$-th cite, so that shifted as in the Remark after Theorem \ref{Th:ShiftOperator} parameters $z_i=\frac{\zeta_i}{\zeta_{i+1}}$, $i=1, \dots, n-1$. 

The qKZ equations which we have described from geometric perspective in \secref{Sec:VertexFunctions} represent a family of difference equations of the following kind
\begin{eqnarray}
p_{i}\Phi=K_i^{(q)}\Phi,\qquad \Phi\in {\rm V}\,,
\label{eq:qKZgenXXZ}
\end{eqnarray}
where as before $p_{i}$ is a shift operator, so that $p_i\Phi(\dots, a_i,\dots)=\Phi(\dots, qa_i\dots)$.
 We refer the reader to \cite{Okounkov:2015aa,Okounkov:2016sya,Aganagic:2017la} for additional technical details on the map $\text{Stab}$, which we briefly reviewed in \secref{Sec:VertexFunctions}.

\subsection{Transfer matrix, Bethe ansatz and Hamiltonians}
\begin{proposition}
The permutation form of the R-matrix is as follows: 
\begin{eqnarray}R(x)=P+\frac{x-x^{-1}}{\hbar^{\frac{1}{2}} x-\hbar^{-\frac{1}{2}}x^{-1}}(I-P^{\hbar}), 
\end{eqnarray}
where 
\begin{equation}
P^{\hbar}=\sum_{a=1}^n e_{aa}\otimes e_{aa}+\hbar^{\frac{1}{2}}\sum^n_{a>b}e_{ab}\otimes e_{ba}+\hbar^{-\frac{1}{2}}\sum_{a<b}^n e_{ab}\otimes e_{ba} 
\end{equation}
and $P=P^{1}$ is a permutation operator.
\end{proposition}
It is very useful to rescale the R-matrix by introducing another notation
\begin{equation}\label{norm}
\widetilde{R}(x)=\frac{\hbar^{\frac{1}{2}} x-\hbar^{-\frac{1}{2}}x^{-1}}{x-x^{-1}}R=I -P^{\hbar}+\frac{\hbar^{\frac{1}{2}} x-\hbar^{-\frac{1}{2}}x^{-1}}{x-x^{-1}}P\,.
\end{equation}
Let us define the transfer matrix $T(x)$, acting on ${\rm V}$, 
obtained by tensoring it with auxiliary cite $V_0$: and then taking the trace of the product of the R-matrices:  
\begin{eqnarray}\label{trma}
T(x)=\text{Tr}_{V_0}\left(\widetilde R_{0{\bf w_{n-1}}}\left(x/a_{\bf w_{n-1}}\right)\dots \widetilde R_{01}\left(x/a_1\right)Z^{(0)}\right)\,,
\end{eqnarray}
where $V_0$ is the $n$-dimensional evaluation representation (auxiliary space).
The following results are known since 1980s using techniques usually referred to as algebraic Bethe ansatz, see e.g. \cite{Zabrodin:}.
\begin{theorem}\label{bethe}
The eigenvalues $\Lambda(x)$ of matrix $T(x)$ are given by the following formula
\begin{equation}
\Lambda(x)=\zeta_1\prod_{k=1}^{{\bf w_{n-1}}}
\frac{x\hbar^{\frac{1}{2}} -a_k\hbar^{-\frac{1}{2}}}{x-a_k}\prod_{\alpha=1}^{{\bf v_1}}\frac{x\hbar^{-\frac{1}{2}}-\sigma^1_{\alpha}\hbar^{\frac{1}{2}}}{x-\sigma^1_{\alpha}}+
\sum^n_{i=2}\zeta_i
\prod_{\alpha=1}^{{\bf v_{i-1}}}\frac{x\hbar^{\frac{1}{2}}-\sigma^{i-1}_\alpha\hbar^{-\frac{1}{2}}}{x-\sigma^{i-1}_\alpha}
\prod_{\alpha=1}^{{\bf v_i}}\frac{x\hbar^{\frac{1}{2}}-\sigma^i_\alpha\hbar^{-\frac{1}{2}}}{x-\sigma^i_\alpha},
 \end{equation}
where ${\bf v_i}$ are integers (so that ${\bf v}_0={\bf v}_{\bf w_{n-1}}=0$) which denote the number of Bethe roots, i.e. the solutions of the following equations:
\begin{eqnarray}
&&\zeta_1\prod_{i=1}^{\bf w_{n-1}}\frac{\sigma^1_{\alpha}\hbar-a_i}{\sigma^1_{\alpha}-a_i}=\zeta_2\prod_{\beta\neq \alpha}^{{\bf v_1}}\frac{\sigma^1_{\alpha}\hbar-\sigma^1_{\beta}}{\sigma^1_{\alpha}-\sigma^1_{\beta}\hbar}\prod^{r_2}_{\beta=1}\frac{\sigma^1_{\alpha}-\sigma^2_{\beta}\hbar}
{\sigma^1_{\alpha}-\sigma^2_{\beta}},\nonumber\\
&&\zeta_i\prod_{\beta =1}^{{\bf v_{i-1}}}\frac{\sigma^i_{\alpha}\hbar-\sigma^{i-1}_{\beta}}{\sigma^1_{\alpha}-\sigma^{i-1}_{\beta}}=\zeta_{i+1}
\prod_{\beta\neq \alpha}^{{\bf v_i}}\frac{\sigma^i_{\alpha}\hbar-\sigma^i_{\beta}}{\sigma^i_{\alpha}-\sigma^i_{\beta}\hbar}\prod^{r_{i+1}}_{\beta=1}\frac{\sigma^i_{\alpha}-\sigma^{i+1}_{\beta}\hbar}{\sigma^i_{\alpha}-\sigma^{i+1}_{\beta}}
\end{eqnarray}
\end{theorem}

One can re-expand the transfer matrix using quantum nonlocal Hamiltonians, which play an important role in the construction (cf. formulae (3.7)-(3.9) of \cite{Zabrodin:}):

\begin{proposition}\label{transf}
$T(x)$ has the following simple pole expansion: 
\begin{eqnarray}
T(x)=C+\frac{\hbar^{\frac{1}{2}}-\hbar^{-\frac{1}{2}}}{2}\sum_{k=1}^{\bf w_{n-1}}\frac{x^2+a_k^2}{x^2-a_k^2}H_k,  
\end{eqnarray}
where $C$, $H_i$ are the operators on ${\rm V}$ and the eigenvalues $h_i$
of $H_i$ are given by the following formula:
\begin{eqnarray}
h_i=\zeta_1\frac{\hbar^{\frac{1}{2}}-\hbar^{-\frac{1}{2}}}{2}\prod_{k\neq i}^{{\bf w_{n-1}}}
\frac{a_i\hbar^{\frac{1}{2}} -a_k\hbar^{-\frac{1}{2}}}{a_i-a_k}\prod_{\alpha=1}^{{\bf v_1}}\frac{a_i\hbar^{-\frac{1}{2}}-\sigma^1_{\alpha}\hbar^{\frac{1}{2}}}{a_i-\sigma^1_{\alpha}}.
\end{eqnarray}
\end{proposition}

\begin{proof}
The first part of the proposition follows from the asymptotic behavior of the transfer matrix $T(x)$, when $x\to {\infty, 0}$ based on the behavior of the R-matrices as well as its trigonometric structure. The second part follows directly from Theorem \ref{bethe}. The operators $H_i$ are known as \textit{nonlocal Hamiltonians} of the XXZ model.
\end{proof}

\begin{proposition}\label{xxzham}
i) The relation between the qKZ operators and the nonlocal Hamiltonians is as follows:
\begin{equation}\label{hi}
H_i=\prod^{\bf w_{n-1}}_{j\neq i}\frac{a_i\hbar^{\frac{1}{2}} -a_j\hbar^{-\frac{1}{2}}}{a_i-a_j}K_i^{(1)}\,.
\end{equation}
ii)The sum of all hamiltonians can be expressed in the following way using Cartan generators acting on ${\rm V}$: 
\begin{eqnarray}
\sum^{\bf w_{n-1}}_{k=1}H_k=\sum_{a=1}^{n} \zeta_a\frac{\hbar^{\frac{S_a}{2}}-\hbar^{-\frac{S_a}{2}}}{\hbar^{\frac{1}{2}}-\hbar^{-\frac{1}{2}}}\,.
\end{eqnarray}
\end{proposition}

\begin{proof} 
Based on the polar structure of the expansion of $T(x)$ in Proposition \ref{transf} we conclude  
that $H_i$ are proportional to residues at points $z=a_k^2$, where $z=x^2$. 
In the definition of the transfer matrix (\ref{trma}) the only term in the product of R-matrices containing such  a pole is $\tilde{R}_{0,k}(x/a_k)$ (this is in accord with the polar structure of the normalized R-matrix (\ref{norm})). 
The corresponding coefficient of the polar term is proportional to the permutation matrix.

Computing the residue and inserting it into the trace of (\ref{trma}) in place of  $\tilde{R}_{0,k}(x/a_k)$ we obtain that 
 $$H_i=\widetilde R_{i\, i-1}\left(\frac{a_i }{a_{i-1}}\right)\dots \widetilde R_{i\,1}\left(\frac{a_i}{a_1}\right)Z^{(i)}\widetilde R_{i\,{\bf w_{n-1}}}\left(\frac{a_i}{a_{\bf w_{n-1}}}\right)\dots \widetilde R_{i\, i+1}\left(\frac{a_i}{a_{i+1}}\right).$$
By removing normalization we obtain (\ref{hi}), thereby proving \textit{i)}. 

The proof of \textit{ii)} goes as follows.  Notice that the asymptotic expansions of the eigenvalues of the transfer matrix are such that 
$$\Lambda(\infty)-\Lambda(0)=(\hbar^{\frac{1}{2}}-\hbar^{-\frac{1}{2}})\sum_k h_k$$ 
from Proposition \ref{transf} and at the same from Bethe ansatz equations   $$\Lambda(\infty)-\Lambda(0)=\sum_{a=1}^{n}\zeta_a(\hbar^{\frac{S_a}{2}}-\hbar^{-\frac{S_a}{2}})\,,$$ thus proving the second part of the proposition.
\end{proof}

\subsection{qKZ vs tRS}
Let $E_J=e_{j_1}\otimes\dots \otimes e_{j_{\bf w_{n-1}}}$, where $J=(j_1, \dots , j_{\bf w_{n-1}})$ is multi-index. Consider 
\begin{equation}
\Phi=\sum_J\Phi_J E_J\in {\rm V}(\{s_a\})\,
\label{eq:PhiSol}
\end{equation}
to be a solution of the qKZ equation of a given weight parametrized by $\{s_a\}$. Let us define $\ell(J)$ as the minimal number of permutations required to put tuple $J=(j_1, \dots. j_{\bf w_{n-1}})$ into $1\le j_1\le j_2\le\dots \le j_{\bf w_{n-1}}\le n$.  
The following proposition holds.
\begin{proposition} \label{eprop} The vector $E^{\hbar}$, such that $E^{\hbar}=\sum_{J}\hbar^{\frac{\ell(J)}{2}} E_{J}$
satisfies the following properties: \\
$$ P^{\hbar}_{i-1\, i}E^{\hbar}=E^{\hbar}, \qquad R_{i-1\,i} (x)E^{\hbar}=P_{i-1\, i}E^{\hbar}\,.$$
\end{proposition}
\begin{proof} One can show that vector $E^{\hbar}$ is obtained as a sum over all possible applications of $P^{\hbar}$ to the standard vector $e_1^{\otimes s_1}\otimes\dots \otimes e_{\bf w_{n-1}}^{\otimes s_{\bf w_{n-1}}}$. This proposition immediately follows. 
\end{proof}

Let us define
\begin{eqnarray}\label{qkztrs}
{\Upsilon}=\langle E^{\hbar}, \Phi\rangle=  \sum_J\hbar^{\frac{\ell(J)}{2}}\Phi_J.
\end{eqnarray}
We will now show that $\Upsilon$ is the eigenfunction of the tRS Hamiltonians.
\begin{theorem}\label{hamfirst} The following formulae hold
\begin{eqnarray}
&&i) \quad p_i\Upsilon=\left\langle E^{\hbar}, K_i^{(1)}\Phi\right\rangle,\\
&&ii) \quad  \sum^{\bf w_{n-1}}_{i=1}\prod^{\bf w_{n-1}}_{i\neq j}\frac{a_i\hbar^{\frac{1}{2}}-a_j\hbar^{-\frac{1}{2}}}{a_i-a_j}\,p_i\Upsilon=\left(\sum^n_{a=1}\zeta_a\frac{\hbar^{\frac{s_a}{2}}-\hbar^{-\frac{s_a}{2}}}{\hbar^{\frac{1}{2}}-\hbar^{-\frac{1}{2}}}\right)\Upsilon\,.
 \end{eqnarray}
\end{theorem}
\begin{proof}
To show $i)$ one needs to prove that $\langle E^{\hbar}, K_i^{(1)}\Phi\rangle$ is independent on $q$. Notice that the dependence on $q$ is involved in the expression for $K_i^{(q)}$ only in the product of R-matrices before multiplication by $Z^{(i)}$. 
Also, one can check that  Proposition \ref{eprop} gives:
$$
\left\langle E^{\hbar}, P^{\hbar}_{i\, i-1}\rho\right\rangle=\left\langle E^{\hbar}, \rho\right\rangle
$$
for any vector $\rho\in {\rm V}(\{s_a\})$. Property $P_{i\, i-1}P^{\hbar}_{i\,i-2}=P^{\hbar}_{i-1\, i-2}P_{i\, i-1}$ together with the second part of Proposition \ref{eprop} reduce that product of R-matrices in $K^{(q)}$ to the product of permutation matrices $P$, i.e.
$$\left\langle E^{\hbar}, R_{i\, i-1}\left(\frac{a_i q}{a_{i-1}}\right)\dots R_{i\,1}\left(\frac{a_iq}{a_1}\right)\rho\right\rangle=\left\langle E^{\hbar}, P_{i\,i-1}\dots P_{i\,1}\rho\right\rangle$$
for any vector $\rho$ from ${\rm V}(\{s_a\})$.

In order to prove $ii)$, we multiply $i)$ by $\frac{a_i\hbar-a_j}{a_i-a_j}$ and sum over $i$. After that we use Proposition \ref{xxzham}.
\end{proof}

The operator, whose eigenfunction we calculated in $ii)$ of the above theorem is the first tRS Hamiltonian.

To show that $\Upsilon$ is indeed an eigenfunction of the entire family of tRS Hamiltonians we need the following statement, which is a generalization of $i)$ of Theorem \ref{hamfirst}.

\begin{proposition}\label{trsxxz}
Action of the products of difference operators $T_q^{(i)}$ on $\Upsilon$ can be expressed as follows:
\begin{eqnarray}
 \prod_{k=1}^d\prod^{\bf w_{n-1}}_{r\neq i_k}\frac{a_{i_k}\hbar-a_r}{a_{i_k}-a_r}\prod^d_{k=1} p_{i_k}\Upsilon=\left\langle E^{\hbar}, H_{i_1}\dots H_{i_d}\Phi\right\rangle  
\end{eqnarray}
\end{proposition}

\begin{proof} 
In order to prove that one just has to use the same principle as in Theorem  \ref{hamfirst} and to prove this identity:
\begin{eqnarray}
\prod^d_{k=1}p_{i_k}\Upsilon=\left\langle E^{\hbar}, K^{(1)}_{i_1}\dots K^{(1)}_{i_d}
\Phi\right\rangle\,,
\end{eqnarray}
namely, use the properties from Proposition \ref{eprop} when moving $q$-shifted R-operators to the left of twisted matrices $Z^{(i)}$.  Then multiplying on the appropriate coefficients as in Proposition \ref{xxzham} we obtain the statement of the theorem.
\end{proof}

Let us use now an important relation proven in \cite{Zabrodin:}:
\begin{proposition}\label{spectral}
The following combinatorial formula holds for the sums of products of Hamiltonians:
\begin{eqnarray}
\sum_{1\le i_1<\dots< i_k\le {\bf w_{n-1}}}
H_{i_1}\dots H_{i_k}\prod_{1\le \alpha<\beta\le k} 
C(a_{i_{\alpha}}/a_{i_{\beta}})=
\left(\frac{\hbar^{\frac{1}{2}}-\hbar^{-\frac{1}{2}}}{2}\right)^k\sum_{1\le i_1<\dots <i_k\le N}\lambda_{i_1}\dots \lambda_{i_k},
\end{eqnarray}
where $$C(x)=\frac{x-x^{-1}}{(x\hbar^{\frac{1}{2}}-x^{-1}{\hbar}^{-\frac{1}{2}})(x\hbar^{-\frac{1}{2}}-x^{-1}{\hbar^{\frac{1}{2}}})}$$
and $\lambda_{i_m}$ are eigenvalues of a certain operator which depend only on 
$\hbar$ and $\{z_i\}$.
\end{proposition}

From now on, to align the notations with the previous section, we rescale parameters in tRS Hamiltonians \eqref{eq:tRSRelationsEl} as $t\to \hbar^{-1}$ and multiply each Hamiltonian $T_d$ by a prefactor
\begin{equation}
\widehat{H}_d = \hbar^{\frac{d}{2}}T_d \,.
\end{equation}

In order to put these Hamiltonians in touch with Proposition \ref{trsxxz}, we need the following statement, which can be proved by direct calculation.

\begin{proposition} The ordered expression for tRS Hamiltonians is given by the following formula:
\begin{eqnarray}
\widehat{H}_d=\sum_{1\le i_1<\dots< i_d\le n}\prod_{k=1}^{d}\prod_{j\neq i_k}\frac{a_{i_k}\hbar^{\frac{1}{2}}-a_j\hbar^{-\frac{1}{2}}}{a_{i_k}-a_j}\prod_{1\le m<n\le d}C(a_{i_m}/a_{i_n})\prod^d_{k=1}p_{i_k} 
\end{eqnarray}
\end{proposition}

Using this expression and then combining Proposition \ref{spectral} with Proposition \ref{trsxxz} we arrive to the main theorem.

\begin{theorem}
Function $\Upsilon$, associated to the solution of qKZ equation via formula (\ref{qkztrs}) is an eigenfunction of tRS Hamiltonians $H_d$.
\end{theorem}

\noindent \textbf{Remark.} We would like to make a comment about the new role which the qKZ equation plays for integrable systems. The above theorem relates solutions of the qKZ equation with solutions of the system of tRS equations. The qKZ equation \eqref{eq:qKZgenXXZ} represents a modern viewpoint on diagonalization of the XXZ spin chain Hamiltonians. 
The connection between  tRS model and the qKZ equation/XXZ chain was first observed in Cherednik \cite{Cherednik:1991mg} and Matsuo \cite{Matsuo1992}. Such correspondence naturally appears on the level of gauge theory, see \cite{Gaiotto:2013bwa}. In \textit{loc. cit.} the tRS model was considered classical (hence the name quantum/classical duality), whereas in this paper we study quantum tRS Hamiltonians, thereby promoting it to quantum/quantum duality between tRS and XXZ models.

\subsection{Geometric Meaning of the qKZ/tRS correspondence}
Finally we shall discuss the geometric outcome of formula \eqref{qkztrs}. First, let us recall the geometric meaning of solutions of qKZ equations. 
The columns of matrix $\widehat\Phi$ \eqref{hatphi} are the solutions of the qKZ equations according to Theorem \ref{Th:ShiftOperator}. 
In order to establish a relationship with solutions $\Phi$ from \eqref{eq:PhiSol} we need to precompose $\widehat\Phi$ with 
$\text{Stab}_+^{-1}$ operator in order to put the shift operator in its R-matrix form according to Theorem \ref{Th:qkzgeom}. The resulting operator $\widetilde \Phi=\text{Stab}_+^{-1}\widehat\Phi$ reads as follows
\begin{equation}\label{solgeom}
\widetilde \Phi=\left\{\widetilde \Phi_{\bf p, \bf q}\right\}=\left\{V_{\bf p}^{((\text{Stab}_{+}^{-1}~ f)_{\bf q})}\cdot\Theta_{\bf p}\right\}\,,
\end{equation}
where we denoted the corresponding insertions of stable basis elements in vertex functions as $(\text{Stab}_{+}^{-1}~ f)_{\bf q}$ and $\Theta_{\bf p}$ are the eigenvalues of matrix $\Theta$ \eqref{hatphi} at fixed points {\bf p}. 

By combining the results of the previous two sections, we can now formulate the following theorem which establishes a geometrical relationship between solutions of qKZ equations $\widetilde\Phi$ and eigenfunctions of tRS Hamiltonians $\mathrm{V}(\textbf{a},\vec{\zeta})$ \eqref{eq:tRSEigen}. 

Recall that after formula \eqref{eq:PhiSol} we defined permutation length $\ell(J)$. Using this notation we can state the following theorem.
\begin{theorem}
Let $X$ be a cotangent bundle to partial flag variety of type A. Then its non-normalized vertex function is given by the following weighted sum of vertices with insertions
\begin{eqnarray}\label{distr}
\mathrm{V}(\textbf{a},\vec{\zeta})=c\cdot\Theta\cdot \sum_{\bf{q}}\hbar^{\frac{\ell({\bf q})}{2}}V^{((\text{Stab}_{+}^{-1}~ f)_{\bf q})}\,,
\end{eqnarray} 
where $c$ is a constant which does not depend on the equivariant parameters and $\ell({\bf q})$ is defined using the natural identification of fixed points with the points in weight subspaces of representations of $\mathfrak{gl}(n)$ and shifting the sign of $z$-variables as prescribed in \secref{Sec:VertexFunctions}.
\end{theorem}

\begin{proof}  
We can identify the components of the solution of the qKZ equations with normalized vertices with insertions as in (\ref{solgeom}) and $\Upsilon$ (\ref{qkztrs}), and the eigenfunctions of tRS Hamiltonians, with the normalized vertices of Theorem \ref{trsvertex}. One can then see that the normalization coefficients are identical. Thus (\ref{qkztrs}) is equivalent to \eqref{distr} for each choice of the contour $C$ defining the vertex.
\end{proof}

\noindent \textbf{Example.} Let us look at $T^*\mathbb{P}^1$ again. The normalized vertex function which is the eigenvalue of the tRS operator $T_1\mathrm{V}=(\zeta_1+\zeta_2)\mathrm{V}$ reads
\begin{equation}
\mathrm{V}=\frac{e^{\frac{\log \zeta_2 \log a_1 a_2}{\log q}}}{2\pi i}\int\limits_{C} \frac{ds}{s}\, \left(\frac{\zeta_1}{\zeta_2}\right)^{-\frac{\log s}{\log q}}\,\prod_{i=1}^2
\frac{\varphi\left(\frac{q}{\hbar}\frac{s}{a_i}\right)}{\varphi\left(\frac{s}{a_i}\right)}\,.
\end{equation}
Let us analyze how the two part of $T_1$ act on $\mathrm{V}$
\begin{equation}
T_1\mathrm{V} = \zeta_2 \frac{e^{\frac{\log \zeta_2 \log a_1 a_2}{\log q}}}{2\pi i}\int\limits_{C} \frac{ds}{s}\left[\frac{ta_1-a_2}{a_1-a_2}\frac{\hbar a_1-s}{a_1-s}+\frac{ta_2-a_1}{a_2-a_1}\frac{\hbar a_2-s}{a_2-s}\right] \left(\frac{\zeta_1}{\zeta_2}\right)^{-\frac{\log s}{\log q}}\,\prod_{i=1}^2
\frac{\varphi\left(\frac{q}{\hbar}\frac{s}{a_i}\right)}{\varphi\left(\frac{s}{a_i}\right)}\,.
\end{equation}
We observe in the square brackets the sum of two terms, which up to a normalization, reproduce stable envelopes for $T^*\mathbb{P}^1$\,. On the other hand, using the Corollary \ref{Th:Lemma4} we get
\begin{align}
T_1 \mathrm{V} &= \zeta_2 \frac{e^{\frac{\log \zeta_2 \log a_1 a_2}{\log q}}}{2\pi i}\int\limits_{C} \frac{ds}{s} \,\prod_{i=1}^2
\frac{\varphi\left(\frac{q}{\hbar}\frac{s}{a_i}\right)}{\varphi\left(\frac{s}{a_i}\right)}\ [T'_1(s)+T'_0(s)]\left(\frac{\zeta_1}{\zeta_2}\right)^{-\frac{\log s}{\log q}}\notag\\
&=\zeta_2\left(\frac{\zeta_1}{\zeta_2}+1\right)V = (\zeta_1+\zeta_2)\mathrm{V}\,,
\end{align}
which illustrates that the elements of the stable basis are summed to a constant.

\bibliography{cpn1}

\end{document}